%% file: main.tex
\documentclass[10pt,oneside]{amsart}
\input{preamble}

\input{notation.tex}
\usepackage[utf8]{inputenc}
\usepackage{adjustbox}
\title{A shadow perspective on equivariant Hochschild homologies }
\author[Adamyk]{Katharine Adamyk}
\address[Adamyk]{Department of Mathematics, University of Western Ontario, London, ON, Canada}
\author[Gerhardt]{Teena Gerhardt}
\address[Gerhardt]{Department of Mathmematics, Michigan State University, East Lansing, MI 48824 }
\author[Hess]{Kathryn Hess}
\address[Hess]{SV UPHESS BMI, \'Ecole Polytechnique F\'ed\'erale de Lausanne, 1015 Lausanne, Switzerland}
\author[Klang]{Inbar Klang}
\address[Klang]{Department of Mathematics, Columbia University, New York, NY 10027. Email inbarklang@gmail.com}
\author[Kong]{Hana Jia Kong}
\address[Kong]{School of Mathematics, Institute for Advanced Study, Princeton, NJ 08540}

\begin{document}

\maketitle
\begin{abstract}
Shadows for bicategories, defined by Ponto, provide a useful framework that generalizes classical and topological Hochschild homology.  In this paper, we define Hochschild-type invariants for monoids in a symmetric monoidal, simplicial model category $\cat V$, as well as for small $\cat V$-categories.  We show that each of these constructions extends to a shadow on an appropriate bicategory, which implies in particular that they are Morita invariant. We also define a generalized theory of Hochschild homology twisted by an automorphism and show that it is Morita invariant. Hochschild homology of Green functors and $C_n$-twisted topological Hochschild homology fit into this framework, which allows us to conclude that these theories are Morita invariant. We also study linearization maps relating the topological and algebraic theories, proving that the linearization map for topological Hochschild homology arises as a lax shadow functor, and constructing a new linearization map relating topological restriction homology and algebraic restriction homology. Finally, we construct a twisted Dennis trace map from the fixed points of equivariant algebraic $K$-theory to twisted topological Hochschild homology.
\end{abstract}



\input{intro}
\input{preliminaries}

\input{shadow}

\input{laxshadow}
\input{trace}

\bibliographystyle{plain}
\bibliography{bib}
\end{document}

%% file: preamble.tex
\usepackage{amssymb, amsmath, amsthm, amsfonts}
\usepackage{mathrsfs,comment}
\usepackage{graphicx}
\usepackage{comment}
\usepackage{todonotes}
\usepackage{placeins}
\usepackage{rotating}
\usepackage{tikz}
\usepackage{float}
\usepackage{hvfloat}
\usepackage{caption}
\usepackage{pdflscape}
\usepackage{notoccite}
\usepackage[bookmarks=true]{hyperref}  
\usepackage{url}
\usepackage[all,arc,2cell]{xy}
\UseAllTwocells
\usepackage{enumerate}
\usepackage{chngcntr}
 \usepackage{lineno}
 \usepackage{blindtext}
\usepackage{verbatim}
\usepackage{soul}
\usepackage[normalem]{ulem}
\usepackage{lscape}
\usepackage{geometry}
\usepackage{pifont}

\usepackage{tikz}
\usepackage{tikz-cd} 
\usepackage[makeroom]{cancel}

\hypersetup{%
  bookmarksnumbered=true,
  colorlinks=true,%
  linkcolor=blue,%
  citecolor=blue,%
  filecolor=blue,%
  menucolor=blue,%
  urlcolor=blue,%
  pdfnewwindow=true,%
  pdfstartview=FitBH}

\def\@url#1{{\tt\def~{\lower3.5pt\hbox{\char'176}}\def\_{\char'137}#1}}

\def\makeautorefname#1#2{\expandafter\def\csname#1autorefname\endcsname{#2}}
\makeautorefname{equation}{Equation}%
\makeautorefname{item}{item}%
\makeautorefname{figure}{Figure}%
\makeautorefname{table}{Table}%
\makeautorefname{part}{Part}%
\makeautorefname{appendix}{Appendix}%
\makeautorefname{chapter}{Chapter}%
\makeautorefname{section}{Section}%
\makeautorefname{subsection}{Section}%
\makeautorefname{subsubsection}{Section}%
\makeautorefname{paragraph}{Paragraph}%
\makeautorefname{subparagraph}{Paragraph}%
\makeautorefname{theorem}{Theorem}%
\makeautorefname{theo}{Theorem}%
\makeautorefname{thm}{Theorem}%
\makeautorefname{addendum}{Addendum}%
\makeautorefname{addend}{Addendum}%
\makeautorefname{add}{Addendum}%
\makeautorefname{maintheorem}{Main theorem}%
\makeautorefname{mainthm}{Main theorem}%
\makeautorefname{corollary}{Corollary}%
\makeautorefname{claim}{Claim}%
\makeautorefname{corol}{Corollary}%
\makeautorefname{coro}{Corollary}%
\makeautorefname{cor}{Corollary}%
\makeautorefname{lemma}{Lemma}%
\makeautorefname{lemm}{Lemma}%
\makeautorefname{lem}{Lemma}%
\makeautorefname{sublemma}{Sublemma}%
\makeautorefname{sublem}{Sublemma}%
\makeautorefname{subl}{Sublemma}%
\makeautorefname{proposition}{Proposition}%
\makeautorefname{proposit}{Proposition}%
\makeautorefname{propos}{Proposition}%
\makeautorefname{propo}{Proposition}%
\makeautorefname{prop}{Proposition}%
\makeautorefname{property}{Property}
\makeautorefname{proper}{Property}
\makeautorefname{scholium}{Scholium}%
\makeautorefname{step}{Step}%
\makeautorefname{conjecture}{Conjecture}%
\makeautorefname{conject}{Conjecture}%
\makeautorefname{conj}{Conjecture}%
\makeautorefname{question}{Question}
\makeautorefname{questn}{Question}
\makeautorefname{quest}{Question}
\makeautorefname{ques}{Question}
\makeautorefname{qn}{Question}
\makeautorefname{definition}{Definition}%
\makeautorefname{defin}{Definition}%
\makeautorefname{defi}{Definition}%
\makeautorefname{def}{Definition}%
\makeautorefname{defn}{Definition}%
\makeautorefname{dfn}{Definition}%
\makeautorefname{notation}{Notation}
\makeautorefname{nota}{Notation}
\makeautorefname{notn}{Notation}
\makeautorefname{remark}{Remark}%
\makeautorefname{rema}{Remark}%
\makeautorefname{rem}{Remark}%
\makeautorefname{rmk}{Remark}%
\makeautorefname{rk}{Remark}%
\makeautorefname{remarks}{Remarks}%
\makeautorefname{rems}{Remarks}%
\makeautorefname{rmks}{Remarks}%
\makeautorefname{rks}{Remarks}%
\makeautorefname{example}{Example}%
\makeautorefname{examp}{Example}%
\makeautorefname{exmp}{Example}%
\makeautorefname{exmps}{Examples}%
\makeautorefname{exam}{Example}%
\makeautorefname{exa}{Example}%
\makeautorefname{algorithm}{Algorith}%
\makeautorefname{algo}{Algorith}%
\makeautorefname{alg}{Algorith}%
\makeautorefname{axiom}{Axiom}%
\makeautorefname{axi}{Axiom}%
\makeautorefname{ax}{Axiom}%
\makeautorefname{case}{Case}%
\makeautorefname{claim}{Claim}%
\makeautorefname{clm}{Claim}%
\makeautorefname{assumption}{Assumption}%
\makeautorefname{assumpt}{Assumption}%
\makeautorefname{conclusion}{Conclusion}%
\makeautorefname{concl}{Conclusion}%
\makeautorefname{conc}{Conclusion}%
\makeautorefname{condition}{Condition}%
\makeautorefname{condit}{Condition}%
\makeautorefname{cond}{Condition}%
\makeautorefname{construction}{Construction}%
\makeautorefname{construct}{Construction}%
\makeautorefname{const}{Construction}%
\makeautorefname{cons}{Construction}%
\makeautorefname{criterion}{Criterion}%
\makeautorefname{criter}{Criterion}%
\makeautorefname{crit}{Criterion}%
\makeautorefname{exercise}{Exercise}%
\makeautorefname{exer}{Exercise}%
\makeautorefname{exe}{Exercise}%
\makeautorefname{problem}{Problem}%
\makeautorefname{problm}{Problem}%
\makeautorefname{probm}{Problem}%
\makeautorefname{prob}{Problem}%
\makeautorefname{solution}{Solution}%
\makeautorefname{soln}{Solution}%
\makeautorefname{sol}{Solution}%
\makeautorefname{summary}{Summary}%
\makeautorefname{summ}{Summary}%
\makeautorefname{sum}{Summary}%
\makeautorefname{operation}{Operation}%
\makeautorefname{oper}{Operation}%
\makeautorefname{observation}{Observation}%
\makeautorefname{observn}{Observation}%
\makeautorefname{obser}{Observation}%
\makeautorefname{obs}{Observation}%
\makeautorefname{ob}{Observation}%
\makeautorefname{convention}{Convention}%
\makeautorefname{convent}{Convention}%
\makeautorefname{conv}{Convention}%
\makeautorefname{cvn}{Convention}%
\makeautorefname{warning}{Warning}%
\makeautorefname{warn}{Warning}%
\makeautorefname{note}{Note}%
\makeautorefname{fact}{Fact}%

  \makeatletter
                   \let\c@lemma\c@theorem
                  \makeatother

\newtheorem{thm}{Theorem}[subsection]
\newtheorem{cor}[thm]{Corollary}
\newtheorem{prop}[thm]{Proposition}
\newtheorem{lem}[thm]{Lemma}

\theoremstyle{definition}
\newtheorem{defn}[thm]{Definition}
\newtheorem{exmp}[thm]{Example}
\newtheorem{exmps}[thm]{Examples}
\newtheorem{rem}[thm]{Remark}
\newtheorem{conj}[thm]{Conjecture}
\newtheorem{rmk}[thm]{Remark}

\newtheorem{notation}[thm]{Notation}

\newtheorem{assumption}[thm]{Assumption}

\makeatletter
\let\c@lem=\c@thm
\let\c@cor=\c@thm
\let\c@prop=\c@thm
\let\c@lem=\c@thm
\let\c@defn=\c@thm
\let\c@exmps=\c@thm
\let\c@rem=\c@thm
\let\c@warn=\c@thm
\let\c@claim=\c@thm
\let\c@quest=\c@thm
\makeatother

\numberwithin{equation}{subsection}

\newcommand{\Z}{\mathbb{Z}}

\newcommand{\R}{\mathbb{R}}

\newcommand{\gmM}{{}^{g}\!\m{M}}

\DeclareSymbolFontAlphabet{\scr}{rsfs}
\newcommand{\sC}{\scr{C}}

\newcommand{\sB}{\scr{B}}
\newcommand{\sR}{\scr{R}}

\newcommand{\K}{\textnormal{K}}
\newcommand{\TR}{\textnormal{TR}}
\newcommand{\Perf}{\mathsf{Perf}}
\newcommand{\End}{\mathsf{End}}
\newcommand{\Aut}{\mathsf{Aut}}

\def\quickop#1{\expandafter\newcommand\csname #1\endcsname{\operatorname{#1}}}
\quickop{Hom} 
\quickop{Tel} \quickop{Mic} 
\quickop{Ext} \quickop{Tor} \quickop{Cotor} \quickop{Id} \quickop{Coker} \quickop{Ker}
\quickop{Lim} \quickop{Colim} \quickop{Holim} \quickop{Hocolim}
\quickop{id} \quickop{tel} \quickop{mic} \quickop{coker} 
\quickop{colim} 
\quickop{hocolim} \quickop{im}
\DeclareMathOperator*{\holim}{holim}

\DeclareMathOperator{\Ho}{Ho}
\DeclareMathOperator{\op}{\operatorname{op}}
\renewcommand{\Bar}{\operatorname{Bar}_\bullet}

\DeclareMathOperator{\THH}{THH}
\DeclareMathOperator{\HH}{HH}

\newcommand{\Mack}{\textrm{Mack}_G}

\newcommand{\m}[1]{\underline{#1}}

\newcommand{\cat}{\mathsf}
\newcommand{\vp}{\varphi}
\newcommand{\ve}{\varepsilon}

\newcommand\adjunction[4]{\xymatrix{#1\ar @<1.18ex>[rr]^{#3}&\perp&#2\ar @<1.18ex>[ll]^{#4}}}

\newcommand\Sh[2]{\lceil {#2}\rceil_{#1}}
\newcommand\tr [1]{\operatorname{tr}(#1)}
\newcommand\mab [3] {{}_{#2}#1_{#3}}
\newcommand\Ob{\operatorname{Ob}}
\renewcommand\Id{\mathrm{Id}}
\newcommand{\Vaut}{\cat V_{\mathrm{aut}}}
\newcommand{\Vautinj}{(\cat V_{\textrm{aut}})_{\textrm {inj}}}
\newcommand{\Vautproj}{(\cat V_{\textrm{aut}})_{\textrm {proj}}}

\definecolor{darkspringgreen}{rgb}{0.09, 0.45, 0.27}
\definecolor{darkterracotta}{rgb}{0.8, 0.31, 0.36}
	\definecolor{darkcoral}{rgb}{0.8, 0.36, 0.27}
	\definecolor{indiagreen}{rgb}{0.07, 0.53, 0.03}
	\definecolor{mountainmeadow}{rgb}{0.19, 0.73, 0.56}
	\definecolor{mountbattenpink}{rgb}{0.6, 0.48, 0.55}
	\definecolor{palatinatepurple}{rgb}{0.41, 0.16, 0.38}
	\definecolor{cinnamon}{rgb}{0.82, 0.41, 0.12}
	\definecolor{chocolate}{rgb}{0.82, 0.41, 0.12}



%% file: notation.tex
\newcommand{\sm}{\wedge}

\renewcommand{\id}{\mbox Id}

\newcommand{\HC}{\textnormal{HC}}



%% file: intro.tex

\section{Introduction}\label{sec:intro}

Theories of Hochschild homology play an essential role in the trace method approach to algebraic $K$-theory. Classical Hochschild homology of a ring $R$ receives a map from the algebraic $K$-theory of $R$, called the Dennis trace,
\[
K_q(R) \to \HH_q(R). 
\]
In this sense, Hochschild homology of a ring serves as an approximation to algebraic $K$-theory. To arrive at a much better approximation, however, one needs to consider a topological analogue of Hochschild homology, called topological Hochschild homology, $\THH(R)$. There is similarly a trace map from algebraic $K$-theory to topological Hochschild homology, which factors the Dennis trace 
\[
K_k(R) \to \pi_k\THH(R) \to \HH_k(R).
\] Through an understanding of the full structure of topological Hochschild homology, one can then define topological cyclic homology (see \cite{BHM,NS}), which often closely approximates algebraic $K$-theory. 

In recent years, equivariant generalizations of Hochschild homology and topological Hochschild homology have been developed. In \cite{AnBlGeHiLaMa} the authors define  $C_n$-twisted topological Hochschild homology for a $C_n$-equivariant ring spectrum $R$, denoted $\THH_{C_n}(R)$, where $C_n$ is the cyclic group of order $n$. Twisted topological Hochschild homology has an algebraic analogue as well, twisted Hochschild homology for Green functors, $\m{\HH}_k^{C_n}$, as defined in \cite{BlGeHiLa}. 

In the classical setting the Hochschild homology of a ring $R$ and its topological analogue, $\THH(R)$ are related by a map
\[
\pi_k \THH(R) \to \HH_k(R).
\]
More generally, for $A$ a $(-1)$-connected ring spectrum this map takes the form
\[
\pi_k \THH(A) \to \HH_k( \pi_0A), 
\]
and is an isomorphism in degree 0. As we will see in Section \ref{sec:laxshadow}, this map is induced by the linearization map $A \to H\pi_0A$ of a $(-1)$-connected ring spectrum, where $H\pi_0A$ denotes the Eilenberg--MacLane spectrum of $\pi_0A$. Hence, this map relating topological Hochschild homology and Hochschild homology is referred to as the linearization map. 
Similarly, in the equivariant setting there is a linearization map relating twisted topological Hochschild homology and the Hochschild homology for Green functors. In particular, for $R$ a $(-1)$-connected $C_n$-ring spectrum, linearization for twisted topological Hochschild homology is a map
\[
\m{\pi}^{C_n}_k \THH_{C_n}(R) \to \m{\HH}_k^{C_n}(\m{\pi}_0R),
\]
where $\m{\pi}^{C_n}_*$ denotes the homotopy Mackey functors for the group $C_n$. In \cite{BlGeHiLa}, the authors construct this linearization map and prove it is an isomorphism in degree 0.

In \cite{aghkk} we developed computational tools to study twisted topological Hochschild homology, including an equivariant B\"okstedt spectral sequence, and computed several examples of twisted THH. In this article, we adopt the framework of bicategorical shadows to deepen our understanding of twisted THH and Hochschild homology for Green functors.

The notion of a shadow on a bicategory, introduced by Ponto in \cite{ponto:shadow},  has proved instrumental in the study of (topological) Hochschild homology, topological restriction homology, algebraic K-theory, and fixed point theory. Both classical Hochschild homology  and its homotopical generalization, topological Hochschild homology of ring spectra and of spectral categories, are examples of shadows \cite{PS},\cite{campbell-ponto}. An advantage of the shadow approach to studying Hochschild theories is that their Morita invariance is immediate once we have established that they are indeed shadows. Viewing rings and modules in a bicategorical framework, Morita equivalence serves as the natural notion of equivalence between 0-cells in a bicategory. As shadows respect  bicategorical structure, they are invariant under Morita equivalences.

The shadow framework also allows for an easy proof that the topological Hochschild homology of a ring spectrum $A$ agrees with the topological Hochschild homology of its category of perfect modules, $\cat{Perf}_A$. One can consider $A$ and $\cat{Perf}_A$ as 0-cells of a bicategory in which they are Morita equivalent. Since topological Hochschild homology is a shadow, the agreement result follows.

It is clear from the examples above that the theory of shadows provides a powerful and useful organizing principle for Hochschild-type theories.  There is therefore good reason to elaborate a general framework in which to define Hochschild homology of ``ring objects" and to show that it is a shadow and thus deduce Morita invariance and an appropriate version of agreement.   In Section \ref{subsec:oneobj} we construct a Hochschild-type invariant of monoids in a nice enough symmetric monoidal, simplicial model category $\cat V$, and show that it extends to a shadow on the bicategory $\sR_{\cat V}$ of (cofibrant) monoids and (homotopy classes of) bimodules, from which we deduce Morita invariance of this generalized Hochschild homology (Corollary \ref{cor:morita}). 

We then extend this Hochschild-type invariant in a natural manner to ``monoids in $\cat V$ with many objects", i.e., small $\cat V$-categories, in Section \ref{subsec:manyobj}. The ``many objects" version of bimodules is formulated in terms of enriched functors and enriched natural transformations.  We show that this invariant of $\cat V$-categories  extends to a shadow on the bicategory $\sC_{\cat V}$ of (pointwise cofibrant) $\cat V$-categories and (homotopy classes of) enriched bimodules.  We then prove that a monoid $A$ in $\cat V$, seen as a $\cat V$-category with one object, is Morita equivalent in $\sC_{\cat V}$ to (a cofibrant replacement of) its category of perfect modules (Proposition \ref{prop:perfmorita}), from which we deduce that the Hochschild homology of $A$ is indeed isomorphic to that of its category of perfect modules (Corollary \ref{cor:agreement}).

In Section \ref{sec:equivHochschild} we develop a construction of Hochschild homology twisted by an automorphism and prove that this new framework, which is not quite a shadow but does satisfy Morita equivalence, encompasses existing notions of equivariant Hochschild homology theories. In particular, we prove that twisted topological Hochschild homology and Hochschild homology for Green functors fit into this framework and hence are Morita invariant. The following theorem combines Theorems \ref{cor-twisted-thh-morita} and \ref{cor-HH-Green-morita}.

\begin{thm}   Twisted topological Hochschild homology, $\THH_{C_n}$, and Hochschild homology for Green functors, $\m{\HH}^{C_n}_H$, satisfy Morita invariance. 
\end{thm}

 We also investigate the linearization map
 \[
\pi_k \THH(R) \to \HH_k(\pi_0 R)
\]
in the context of bicategorical shadows. In particular, we show in Proposition \ref{prop-EM-lax} that this linearization map arises from a lax shadow functor between these bicategorical shadows. Lax shadow functors, defined in \cite{PS}, can be thought of as structure-preserving natural transformations between bicategorical shadows.

This perspective on linearization maps also allows us to define a new linearization map, from topological restriction homology ($\TR$) to algebraic restriction homology, $tr$. This algebraic version of $\TR$, $tr$, was defined in \cite{BlGeHiLa} using Hochschild homology for Green functors. We prove that $tr$ is a bicategorical shadow and conclude that it is Morita invariant. We also show that our new linearization map from $\TR$ to $tr$ arises from a lax shadow functor. The following is proven as Propositions \ref{prop-tr-shadow} and \ref{prop-tr-lax} and Corollary \ref{cor-lin-map}.

\begin{thm} For a $(-1)$-connected ring spectrum $R$, there is a linearization map \[\pi_k \TR(R) \to tr_k(\pi_0(R)).\]  This linearization map arises from a lax shadow functor.

\end{thm}

One of the central properties of Hochschild theories is that they receive trace maps from algebraic K-theory. The Dennis trace map from algebraic K-theory to Hochschild homology factors through $\THH$. It was shown in \cite{BlGeHiLa} that the Dennis trace map also factors through Hochschild homology for Green functors and through algebraic $tr$. 

It is natural to ask whether twisted topological Hochschild homology, $\THH_{C_n}$, also receives a trace map from a version of algebraic K-theory. In Theorem \ref{thm-dennis-trace} and Corollary \ref{cor:twisted}, we prove that indeed it receives a trace map from the fixed points of equivariant algebraic K-theory, $\K_{C_n}$. Like the classical Dennis trace $\K \to \THH$, this map is given by taking the trace of an endomorphism (for example, the trace of a matrix.)

\begin{thm}Let $R$ be a $C_n$-ring spectrum. There is a twisted Dennis trace map $$\K_{C_n}(R)^{C_n} \to \THH_{C_n}(R).$$
\end{thm}

For a prime $p$ that does not divide $n$, $\THH_{C_n}(R)$ has a $p$-cyclotomic structure \cite{AnBlGeHiLaMa}. This can be used to define twisted topological restriction homology, $\TR_{C_n}(R;p)$. We use work on algebraic K-theory from the shadow perspective \cite{CLMPZ} to prove that when $p$ is prime to $n$, the trace map $\K_{C_n}(R)^{C_n} \to \THH_{C_n}(R)$ lifts to $\TR_{C_n}(R;p)$.

\subsection{Organization}
In Section \ref{sec:background}, we cover necessary background on bicategories and shadows. In Section \ref{sec:shadow}, we define a generalized Hochschild-type invariant and prove that it is a shadow; we conclude that it satisfies Morita invariance. In Section \ref{sec:equivHochschild}, we develop a theory of Hochschild homology twisted by an automorphism and use this generalized framework to prove that $\THH_{C_n}$, and $\m{\HH}^{C_n}$ are Morita invariant. In Section \ref{sec:laxshadow}, we study linearization maps from the perspective of lax shadow functors, including a new linearization map from $\TR$ to algebraic $tr$. In Section \ref{sec:trace}, we show that $\THH_{C_n}$ receives a trace map from the fixed points of equivariant algebraic K-theory, and that this trace map factors through $\TR_{C_n}$.

\subsection{Notation and Conventions}

The indexing universe for equivariant spectra is always understood to be complete.

\subsection{Acknowledgements} 
This research grew out of a Women in Topology III project. We are grateful to the organizers of the Women in Topology III workshop, as well as to the Hausdorff Research Institute for
Mathematics, where much of this research was carried out. We are grateful to Gabe Angelini-Knoll, Asaf Horev, Cary Malkiewich, Mona Merling, Kate Ponto, and Inna Zakharevich for helpful conversations and for sharing drafts of their work. In addition, we thank the authors of \cite{CLMPZ} for sharing their proof of Proposition \ref{prop-TR-shadow}. We thank Emanuele Dotto for pointing out an error in a previous draft. We also thank the anonymous referees for helpful comments. The second author was supported by National Science Foundation grants DMS-1810575, DMS-2104233, and DMS-2052042. The fifth author was  supported by National Science
Foundation grant DMS-1926686. Some of this work was done while the second author was in residence at the Mathematical Sciences Research Institute in Berkeley, CA (supported by the National Science Foundation under grant DMS-1440140) during the Spring 2020 semester. The Women in Topology III workshop was supported by NSF
grant DMS-1901795, the AWM ADVANCE grant NSF HRD-1500481, and Foundation Compositio
Mathematica.

%% file: preliminaries.tex

\section{Background}\label{sec:background}
\subsection{Bicategorical preliminaries}

Recall that a \emph{bicategory} $\sB$ consists of a class $\sB_0$ of \emph{0-cells} (also called \emph{objects}), together with a category $\sB(A,B)$ for every pair $A,B$ of 0-cells.  The objects and morphisms of the category $\sB (A,B)$ are called \emph{$1$-cells} and \emph{$2$-cells}, respectively.  For every zero cell $A$, there is a distinguished \emph{identity 1-cell}, $U_A$, in $\sB(A,A)$. For every triple $A,B,C$ of 0-cells, there is a \emph{composition functor}
\[ 
-\odot -: \sB (A,B) \times \sB (B,C) \to \sB (A,C),
\]
along with natural isomorphisms 
\begin{align*}
a: (M\odot N) \odot P \xrightarrow{\cong} & M \odot (N \odot P)\\
l: U_A \odot M \xrightarrow{\cong} & M\\
r: M \odot U_B \xrightarrow{\cong} & M
\end{align*}
which satisfy the same coherence axioms as those for a monoidal category.
We refer the reader to \cite[Definition 1.2.1]{leinster:bicat} for a complete definition.

\begin{notation}
We often write $\mab MAB$ to denote a 1-cell $M$ in $\sB(A,B)$.
\end{notation}

Well-known examples of bicategories include the bicategory $\sR$ of (ordinary) rings, in which the 0-cells are rings, while $\sR(A,B)$ is the category of $(A,B)$-bimodules for any pair of rings $A,B$, and the composite of an $(A,B)$-bimodule $M$ and a $(B,C)$-bimodule $N$ is their tensor product over $B$, $M\otimes_B N$.  The bicategory $\sC$ of small categories, functors, and natural transformations is another useful example to keep in mind.

Generalizing simultaneously the notions of a dualizable bimodule and its dual and of an adjoint pair of functors, we have the following definition of dual pairs of 1-cells in a bicategory.

\begin{defn} Let $\sB$ be a bicategory, and let $A,B\in \sB_0$. A pair of 1-cells $(\mab MAB,\mab NBA)$ is a \emph{dual pair for $(A,B)$} if there exist 2-cells $\ve\colon N\odot M \to  U_B$ and  $\eta\colon U_A \to  M\odot N$, called \emph{evaluation} and \emph{coevaluation}, respectively, satisfying the triangle identities, i.e., the composites of 2-cells
\[
N\cong N\odot U_A \xrightarrow{N\odot\eta} N\odot (M\odot N)\cong (N\odot M)\odot N \xrightarrow{\ve\odot N} U_B\odot N \cong N
\]
and
\[
M\cong U_A\odot M \xrightarrow{\eta\odot M} (M\odot N)\odot M\cong M\odot (N\odot M) \xrightarrow{M\odot \ve} M\odot U_B \cong M
\]
are identities.
\end{defn}

The standard notion of equivalence of 0-cells, formulated below in terms of dual pairs, generalizes the notions of both Morita equivalence of rings and equivalence of categories. 

\begin{defn}\label{def:Morita}
Let $\sB$ be a bicategory. A dual pair $(\mab MAB, \mab NBA)$ is a \emph{Morita equivalence} between $A$ and $B$ if $(\mab NBA, \mab MAB)$ is also a dual pair, and the evaluation morphism of each dual pair is an isomorphism, with inverse the coevaluation morphism of the other dual pair. 

Let $A$ and $B$ be 1-cells of $\sB$.  If there exists a Morita equivalence $(\mab MAB, \mab NBA)$, then $A$ and $B$ are \emph{Morita equivalent}.
\end{defn}

\subsection{The shadow framework}

In \cite{ponto:shadow}, Ponto formulated the following definition, which has proved crucial to developing a global understanding of Hochschild homology and its properties, in all its variants and generalizations.

\begin{defn}\label{defn:shadow}  Let $\cat C$ be a category. A \emph{$\cat C$-shadow} on a bicategory $\sB$ consists of a family of functors
\[
\big\{\Sh{A}{-}\colon \sB(A,A) \to \cat C\mid A\in \sB_0\big\}
\]
together with natural isomorphisms
\[
\theta_{A,B}\colon \Sh{A}{M\odot N} \xrightarrow \cong \Sh{B}{N\odot M},
\]
for all 1-cells $\mab MAB$ and $\mab NBA$ and all 0-cells $A$ and $B$, that are appropriately associative and unital. In particular, the following diagrams commute, where $a$ is the associator and $u$ the unitor, and $l$ and $r$ are the left and right units, respectively,  of the bicategory.
\[
\xymatrix{\Sh{}{(M\odot N)\odot P)} \ar[r]^{\theta}\ar[d]_{\Sh{}{a}}& \Sh{}{P\odot (M\odot N))}
 \ar[r]^{\Sh{}{a}} & \Sh{}{(P\odot M)\odot N)}  \\
\Sh{}{M\odot (N\odot P)} \ar[r]^{\theta} & \Sh{}{(( N\odot P)\odot M}  \ar[r]^{\Sh{}{a}} & \Sh{}{ N\odot (P\odot M)} \ar[u]_{\theta}
},
\]
\[
\xymatrix{\Sh{}{M\odot U_A} \ar[r]^{\theta} \ar[dr]_-{\Sh{}{r}} & \Sh{}{U_A \odot M} \ar[r]^{\theta} \ar[d]^{\Sh{}{l}} & \Sh{}{M \odot U_A} \ar[dl]^{\Sh{}{r}} \\
& \Sh{}{M} &
}.
\]

\end{defn}

When the source of a shadow component $\Sh{A}{-}\colon \sB(A,A) \to \cat C$ is clear from context, we usually suppress it from the notation.

Let $\cat{SH}$ denote the stable homotopy category. Campbell and Ponto proved in \cite{campbell-ponto} that topological Hochschild homology is an $\cat{SH}$-shadow on the bicategory $\sR_{\cat{Sp}}$ whose objects are structured ring spectra (i.e., monoids in some nice monoidal model category of spectra, $\cat{Sp}$) and whose morphism categories are the homotopy categories of categories of spectral bimodules.

Shadows satisfy several interesting and useful properties that follow essentially from their bicategorical nature, in particular related to dual pairs of 1-cells. We recall here from \cite{ponto:shadow} those that are most relevant for our work.

\begin{defn} \cite{ponto:shadow}. Let $\sB$ be a bicategory equipped with a $\cat{C}$-shadow $\Sh{}{-}$, and let $(\mab MAB,\mab NBA)$ be a dual pair.  Let $\mab PBB$ and $\mab QAA$ be 1-cells.

The \emph{trace} of a 2-cell $f\colon Q\odot M \to  M\odot P$ with respect to the $\cat C$-shadow $\Sh{}{-}$ is the morphism $\tr f\in \cat C\big (\Sh{} Q, \Sh {} P\big)$ given by the composite
\[
\Sh{}{Q} \xrightarrow{\Sh{}{Q\odot \eta}} \Sh {}{Q\odot M\odot N}\xrightarrow{\Sh {}{f\odot N}} \Sh {}{M\odot P\odot N}\cong \Sh {}{N\odot M\odot P} \xrightarrow{\Sh {}{\ve\odot P}}  \Sh {}P
\]
where we have suppressed unitor and associator isomorphisms.

The trace of a 2-cell $g\colon N\odot Q\to  P\odot N$, which is also a  morphism $\tr g\in \cat C \big(\Sh{} Q, \Sh{} P\big)$, is defined analogously.  
\end{defn}

We distinguish the following important special case of the trace.
\begin{defn}
Let $\sB$ be a bicategory equipped with a $\cat C$-shadow $\Sh{}{-}$, and let $(\mab MAB,\mab NBA)$ be a dual pair.

The \emph{Euler characteristic} of $M$, denoted $\chi(M)$, is the trace of $\Id _M$, the identity 2-cell on $M$, i.e., $\chi(M)$ is the composite
\[
\Sh{}{U_A} \xrightarrow{\Sh{}{\eta}}\Sh{}{M\odot N}\cong \Sh{}{N\odot M}\xrightarrow {\Sh{}{\ve}}\Sh{}{U_B}.
\]
\end{defn}

Traces behave particularly well in the context of Morita equivalences.  For any dualizable pair $(\mab MAB, \mab NBA)$ and 1-cells $\mab QAA$ and $\mab PBB$, consider the 2-cells
\[
\eta\odot Q\odot M\colon Q\odot M \to M\odot N\odot Q\odot M
\]
and
\[
M\odot P\odot \ve\colon M\odot P\odot N\odot M \to M\odot P,
\]
where we are suppressing unitors and associators.

\begin{prop}\label{prop:morita}\cite{campbell-ponto} Let $\sB$ be a bicategory equipped with a $\cat C$-shadow.  Let $(\mab MAB, \mab NBA)$ be a Morita equivalence.

For any 1-cell $\mab QAA$, 
\[
\tr {\eta\odot Q\odot M}\colon \Sh{}Q \to \Sh{}{N\odot Q\odot M}
\]
is an isomorphism with inverse
\[
\tr {N\odot Q\odot \ve}\colon \Sh{}{N\odot Q\odot M}\to \Sh {}{Q}.
\]
In particular, the Euler characteristic $\chi (M)\colon \Sh{}{U_A} \to \Sh{}{U_B}$ is an isomorphism.
\end{prop}

Note that the last statement of the proposition above follows from the fact that the Euler characteristic $\chi(N)$ of the right member $N$ of the dual pair $(M,N)$ provides an inverse to $\chi(M)$, by the definition of Morita equivalence.

%% file: shadow.tex

\section{Hochschild shadows}\label{sec:shadow}

Hochschild homology theories form an important family of shadows. Classical Hochschild homology extends to a shadow (see Example 6.4 of \cite{PS}.) In Theorem 2.17 of \cite{campbell-ponto}, Campbell and Ponto proved that topological Hochschild homology of ring spectra and of spectral categories does so as well. In this section, we generalize these results to any nice enough symmetric monoidal, simplicial model category.

Throughout this section, let $(\cat V, \otimes, S)$ denote a symmetric monoidal, simplicial model category, i.e., a simplicial model category equipped with a symmetric monoidal structure with respect to which it is a also a monoidal model category. In particular, we assume that $\cat V$ is closed monoidal.

\subsection{The one-object case}\label{subsec:oneobj}
 Here we construct a Hochschild-type invariant of monoids in a symmetric monoidal, simplicial model category $\cat V$, and show that it extends to a shadow, from which we deduce Morita invariance of this generalized Hochschild homology.

\begin{assumption}\label{assumption:oneobj}
 We assume henceforth that  $(\cat V, \otimes, S)$ is cofibrantly generated and has functorial fibrant replacements and that for any pair of monoids $A$ and $B$ in $\cat V$, the category $\mab {\cat {Mod}}AB$ of $(A,B)$-bimodules in $\cat V$ admits a model structure right-induced from that of $\cat V$, i.e., such that the weak equivalences and fibrations are created in $\cat V$. This hypothesis holds if, for example, every object in $\cat V$ is small relative to $\cat V$, and $\cat V$ satisfies the monoid axiom \cite [Theorem 4.1]{schwede-shipley}.
\end{assumption}

Since $\cat V$ is a simplicial model category, there is a \emph{homotopy colimit} functor
\[
\hocolim_{\Delta^{\mathrm{op}}}\colon \cat V^{\Delta^{\mathrm{op}}} \to \cat V;
\]
see \cite[Definition 18.1.2]{hirschhorn}.

\begin{defn} \label{def:bicat-V}
Let $\sR_{\cat V}$ denote the bicategory defined as follows.
\begin{itemize}
    \item The 0-cells of $\sR_{\cat V}$ are monoids in $\cat V$ whose underlying object in $\cat V$ is cofibrant.
    \item For any monoids $A$ and $B$ in $\cat V$, the category $\sR_\cat V(A,B)$ is $\Ho (\mab {\cat {Mod}}AB)$, the homotopy category of the category of $(A,B)$-bimodules in $\cat V$, which we view explicitly as the category with bifibrant objects of $\mab {\cat {Mod}}AB$ as objects and homotopy classes of morphisms in $\mab {\cat {Mod}}AB$ as morphisms.
    \item Given objects $\mab MAB$ in $\Ho (\mab {\cat {Mod}}AB)$ and $\mab NBC$ in $\Ho (\mab {\cat {Mod}}BC)$, their composite $M\odot N$ is defined to be their derived tensor product over $B$.  More explicitly, we set
    \[
    M\odot N = (\hocolim_{\Delta^{\mathrm{op}}} \Bar(M;B;N))^f,
    \]
    the fibrant replacement of the homotopy colimit of the bar construction $\Bar(M;B;N)$.
\end{itemize}
\end{defn}

\begin{notation}
  Henceforth, if $X_\bullet$ is a simplicial object in a simplicial model category $\cat M$, then we abuse notation and terminology somewhat and let $|X_\bullet|$ denote $(\hocolim_{\Delta^{\mathrm{op}}} X_\bullet)^f$, which we refer to as the geometric realization of $X_\bullet$.  If all simplicial objects in $\cat M$ are Reedy cofibrant (e.g., if $\cat M$ is a topos, such as $\mathsf{sSet}$, or an additive category \cite{rezk:MO}), then  $(\hocolim_{\Delta^{\mathrm{op}}} X_\bullet)^f$ is weakly equivalent to the usual geometric realization. 
\end{notation}
  
It is not hard, if a bit tedious, to check that $\sR_{\cat V}$ is indeed a bicategory. In particular, for any 0-cell $A$ in $\sR_{\cat V}$, the unit object $U_A$ in $\sR_{\cat V}(A,A)$ is just $A$ seen as a bimodule over itself in the canonical way. Note that Assumption \ref{assumption:oneobj} and the cofibrancy condition on the 0-cells together ensure that horizontal composition is well defined by \cite[Definition 18.5.3]{hirschhorn}.

\begin{defn}
Let $A$ be a monoid in $\cat V$, and let $M$ be an $A$-bimodule. 
The \emph{Hochschild construction} on $A$ with coefficients in $M$, denoted $\HH_{\cat V} (A;M)$, is $|B_{\bullet}^{\mathrm {cyc}}(A;M)|$, the geometric realization of the cyclic bar construction 
(see, e.g.  \cite[Section 2]{BHM}) on $A$ with coefficients in $M$.
\end{defn}

\begin{prop}
\label{prop:HH_is_shadow}
The Hochschild construction $\HH_{\cat V}$ extends to a $\Ho \cat V$-shadow on $\sR_{\cat V}$.
\end{prop}

\begin{proof}
The Hochschild construction clearly defines a functor 
\[
\HH_{\cat V} (A;-)\colon \Ho (\mab{\cat {Mod}(\cat V)}AA) \to \Ho \cat V
\]
for every $A$. Moreover, a slight generalization of the Dennis--Morita--Waldhausen argument in \cite[Proposition 6.2]{blumberg-mandell:localization} shows that there is a natural isomorphism
\[
\theta_{A,B}:\HH_{\cat V} (A;M\odot N) \cong \HH_{\cat V} (B;N\odot M)
\]
for all $\mab MAB$ and $\mab NBA$.  Indeed, since all objects are cofibrant, the left-hand side is weakly equivalent to 
$$\hocolim_{\Delta^{\mathrm{op}}\times\Delta^{\mathrm{op}} } B_{\bullet}^{\mathrm {cyc}}\big(A; \Bar(M;B;N)\big)$$
while the right-hand side is weakly equivalent to 
$$\hocolim_{\Delta^{\mathrm{op}}\times\Delta^{\mathrm{op}} } B_{\bullet}^{\mathrm {cyc}}\big(B; \Bar(N;A;M)\big).$$
Since the two bisimplicial objects of which we take the homotopy colimit are in fact isomorphic, we can conclude. 

It is straightforward to check the remaining shadow conditions.  The required relations between the associator and $\theta$ and between the unitor and $\theta$ follow immediately from associativity and unit coherence for the symmetric monoidal structure on $\cat V$.  
\end{proof}

\begin{rmk}\label{rmk-HH-shadow}
Taking $\cat V$ to be the category of simplicial abelian groups, we recover the result that Hochschild homology is a shadow on the bicategory $\sR_{s\cat{Ab}}$ of simplicial rings and their bimodules.
\end{rmk}

Together with Proposition \ref{prop:morita}, the proposition above implies that the following version of Morita invariance holds.

\begin{cor}\label{cor:morita} Let $(\mab MAB, \mab NBA)$ be a Morita equivalence in $\sR_{\cat V}$.  For any $A$-bimodule $Q$, there is an isomorphism
\[
\HH_{\cat V}(A;Q) \xrightarrow{\cong} \HH _{\cat V}(B;N\odot Q\odot M).
\]
In particular, the Euler characteristic
\[
\chi(M)\colon \HH_{\cat V}(A;A) \to \HH_{\cat V}(B;B)
\]
is an isomorphism.
\end{cor}

\subsection{The many-object case}\label{subsec:manyobj}

The constructions above can be generalized in a natural manner to ``monoids in $\cat V$ with many objects", i.e., small $\cat V$-categories, which are the objects of a category that we denote $\cat {VCat}$. The ``many objects" version of bimodules is formulated in terms of enriched functors and enriched natural transformations.  We begin by fixing notation and recalling a few key constructions from enriched category theory, for which we recommend the references \cite{kelly} and \cite{riehl}.

For any $\cat V$-category $\cat A$ and any $a,a'\in \Ob \cat A$, let $\cat A(a,a')$ denote the $\cat V$-object of morphisms from $a$ to $a'$.  Given $\cat V$-categories $\cat A$ and $\cat B$, let $\cat {VFun}(\cat A, \cat B)$ denote the (ordinary) category of $\cat V$-functors from $\cat A$  to $\cat B$ and of $\cat V$-natural transformations between them. Let $\cat S$ denote the  $\cat V$-category with one object $*$ and such that $\cat S(*,*)=S$, the unit object in $\cat V$.  

Given $\cat V$-categories $\cat A$, $\cat B$, and $\cat C$, and a $\cat V$-functor $F: \cat A \to \cat B$, there is an induced adjunction
\[
\adjunction{\cat {VFun}(\cat A, \cat C)}{\cat {VFun}(\cat B, \cat C)}{F_!}{F^*},
\]
where $F^*$ is given by precomposition with $F$, and $F_!$ is enriched left Kan extension. Explicitly, for all $\cat V$-functors $G:\cat A \to \cat C$ and all $b\in \Ob \cat B$, the enriched left Kan extension of $G$ along $F$ evaluated at $b$ can be computed as an enriched coend \cite[7.6.7]{riehl}
$$F_!(G)(b)=\int ^{\cat A}\cat B\big( F(a), b\big)\otimes G(a).$$

\begin{defn}
Let $\cat A$ and $\cat B$ be $\cat V$-categories.  An \emph{$(\cat A,\cat B)$-bimodule} is a $\cat V$-functor 
\[
M\colon \cat A^{\mathrm{op}} \otimes \cat B \to \cat V,
\]
where $\cat A^{\mathrm{op}} \otimes \cat B$ denotes the $\cat V$-category with class of objects $\Ob \cat A \times \Ob \cat B$, where we let $a\otimes b$ denote the pair $(a,b)\in\Ob \cat A \times \Ob \cat B$, to avoid conflicting notation.  The morphism $\cat V$-objects of $\cat A^{\mathrm{op}} \otimes \cat B$ are specified by 
\[
\cat A^{\mathrm{op}} \otimes \cat B\big( a\otimes b, a'\otimes b'\big) =\cat A(a',a)\otimes B(b,b').
\]
A \emph{morphism of $(\cat A,\cat B)$-bimodules} is a $\cat V$-natural transformation between the corresponding $\cat V$-functors.
\end{defn}

In other words, the category $\mab {\cat {Mod}}{\cat A}{\cat B}$ of $(\cat A, \cat B)$-bimodules is exactly $\cat{VFun}(\cat A^{\mathrm{op}}\otimes \cat B, \cat V)$. In particular, $\mab {\cat {Mod}}{\cat S}{\cat S}$ is isomorphic as a $\cat V$-category to $\cat V$.

\begin{exmp}\label{exmp:bimodule}
Let $\cat A$, $\cat B$, and $\cat C$ be $\cat V$-categories. Every pair of $\cat V$-functors, $F\colon \cat A\to \cat C$ and $G\colon \cat B \to \cat C$,  gives rise to an $(\cat A,\cat B)$-bimodule
\[
\mab {\cat C}FG\colon \cat A^{\mathrm{op}} \otimes \cat B \to \cat V : a\otimes b \mapsto \cat C\big( F(a), G(b)\big).
\]
An important special case of this construction is 
\[
\widehat{\cat A}=\mab {\cat A}\Id\Id\colon \cat A^{\mathrm{op}} \otimes \cat A \to \cat V : a\otimes a' \mapsto \cat A(a, a'),
\]
the canonical $\cat A$-bimodule structure on $\cat A$.
\end{exmp}

\begin{rmk}
We can also use enriched coends to define a many-objects version of tensoring two bimodules over a common coefficient base.  Let $\cat A$, $\cat B$, and $\cat C$ be $\cat V$-categories.  There is a $\cat V$-functor
$$-\otimes _{\cat B}-:\mab {\cat {Mod}}{\cat A}{\cat B}\otimes \mab {\cat {Mod}}{\cat B}{\cat C}\to \mab {\cat {Mod}}{\cat A}{\cat C}$$
specified on objects by 
$$M\otimes _{\cat B}N:\cat A^{\mathrm{op}} \otimes \cat C \to \cat V: a\otimes c \mapsto \int ^{\cat B} M(a\otimes -)\otimes N(- \otimes c)$$
for all $(\cat A, \cat B)$-bimodules $M$ and $(\cat B, \cat C)$-bimodules $N$.
\end{rmk}

There are also ``many objects" versions of the bar construction and cyclic bar construction (see, e.g.,  \cite{Mit72} for additive categories and \cite{blumberg-mandell:localization} for spectral categories), constructed from the following building blocks.  Let $\cat A$ and $\cat B$ be $\cat V$-categories, and let $M$ be an $(\cat A, \cat B)$-bimodule.  Let $\widehat {\cat A} \boxtimes M$ and $M\boxtimes \widehat{\cat B}$ denote the $(\cat A, \cat B)$-bimodules specified on objects by
$$(\widehat {\cat A} \boxtimes M)(a\otimes b)=\coprod_{a'\in \Ob \cat A}\widehat {\cat A}(a\otimes a')\otimes M(a'\otimes b)=\coprod_{a'\in \Ob \cat A}\cat A(a,a')\otimes M(a'\otimes b)$$
and
$$(M\boxtimes \widehat {\cat B})(a\otimes b)=\coprod_{b'\in \Ob \cat B}M(a\otimes b')\otimes \widehat {\cat B}(b'\otimes b)=\coprod_{b'\in \Ob \cat B}M(a\otimes b')\otimes \cat B(b', b).$$

There are $\cat V$-natural transformations
$$\lambda: \widehat{\cat A} \boxtimes M\to M \quad\text{ and }\quad \rho: M\boxtimes \widehat{\cat B} \to M$$
with components given by the following composites
$$\coprod_{a'\in \Ob \cat A}\cat A(a,a')\otimes M(a'\otimes b) \to \coprod_{a'\in \Ob \cat A}{\cat V}\big(M(a'\otimes b), M(a\otimes b)\big)\otimes M(a'\otimes b)\to M(a\otimes b)$$
and
$$\coprod_{b'\in \Ob \cat B}M(a\otimes b')\otimes \cat B(b', b) \to \coprod_{b'\in \Ob \cat B}M(a\otimes b')\otimes \cat V\big(M(a\otimes b'), M(a\otimes b)\big)\to M(a\otimes b),$$
where the first morphism in the composite uses the $\cat V$-functor structure of $M (-\otimes b)$, which is contravariant, and of $M(a\otimes -)$, respectively, and the second is built from ``evaluation" morphisms.  In a slight abuse of language, we call $\lambda$ and $\rho$ the \emph{left action of $\cat A$} and the \emph{right action of $\cat B$} on $M$.

\begin{defn}\label{defn:bar-multiobj}
Let $\cat A$, $\cat B$, and $\cat C$ be $\cat V$-categories. Let $M\in \Ob \mab {\cat {Mod}}{\cat A}{\cat B}$ and $N\in \Ob \mab {\cat {Mod}}{\cat B}{\cat C}$.  The \emph{simplicial bar construction} on $\cat B$ with coefficients in $M$ and $N$ is the simplicial $(\cat A, \cat C)$-bimodule $\Bar(M;\widehat{\cat B};N)$ specified by 
$$\operatorname{Bar}_n (M;\widehat{\cat B};N)= M \boxtimes \widehat{\cat B}^{\boxtimes n}\boxtimes N,$$
with faces given by the right action of $\cat B$ on $M$, composition in $\cat B$, and the left action of $\cat B$ on $N$.  Degeneracies are given by the unit of $\cat B$.

Similarly, for $P\in \Ob \mab {\cat {Mod}}{\cat A}{\cat A}$, the \emph{cyclic bar construction} on $\cat A$ with coefficients in $P$ is the cyclic object $B_\bullet^{\mathrm {cyc}}(\widehat{\cat A};P)$ in
$\mab {\cat {Mod}}{\cat S}{\cat S}$ (which can therefore also be seen as a cyclic object in $\cat V$)
that is specified by
$$B_n^{\mathrm {cyc}} (\cat A;P) = P\boxtimes \widehat{\cat A}^{\boxtimes n},$$
with faces given by the right action of $\cat A$ on $P$, composition in $\cat A$, and the left action of $\cat A$ on $P$ together with a cyclic permutation of the factors.  Degeneracies are given by the unit of $\cat A$.
\end{defn}

\begin{rmk}
The cyclic bar construction is a simplicial object in $\mab {\cat {Mod}}{\cat S}{\cat S}$ rather than $\mab {\cat {Mod}}{\cat A}{\cat A}$ because the cyclic permutation applied in the last face at each level is in general not compatible with the $\cat A$-action on either side.
\end{rmk}

\begin{assumption}\label{assumption:manyobj}
We assume henceforth that the unit $S$  of $\cat V$ is cofibrant and that for any pair of $\cat V$-categories $\cat A$ and $\cat B$, the category $\mab {\cat {Mod}}{\cat A}{\cat B}$ of $(\cat A,\cat B)$-bimodules in $\cat {VCat}$ admits a simplicial model structure in which the weak equivalences and fibrations are determined objectwise.  This hypothesis holds if, for example, $\cat V$ is a \emph{locally presentable base} \cite[Definition 2.1]{moser}, and tensoring with $\cat A^{\mathrm{op}}\otimes \cat B\big( a\otimes b, a'\otimes b'\big)$ preserves acyclic cofibrations in $\cat V$ for every pair of $\cat V$-categories $\cat A$ and $\cat B$. These conditions guarantee the existence of a projective model structure on the category of $(\cat A, \cat B)$-bimodules \cite[Remark 4.5]{moser}.   Moreover, the proof of \cite[Theorem 11.7.3]{hirschhorn} applies essentially verbatim in this case, to show that objectwise tensoring and cotensoring by simplicial sets endows  $\mab {\cat {Mod}}{\cat A}{\cat B}$ with the structure of a simplicial model category.

Examples of locally presentable bases include the categories of simplicial sets, of symmetric spectra, and of chain complexes over a commutative ring \cite[Examples 5.6, 6.6, 6.7]{moser}.  The condition on preservation of acyclic cofibrations holds if, for example, all objects in $\cat V$ are cofibrant or we consider only those $\cat V$-categories such that the morphism objects are cofibrant in $\cat V$.
\end{assumption}

\begin{defn}\label{defn:CcatV}
Let $\sC_{\cat V}$ denote the bicategory defined as follows.
\begin{itemize}
    \item The 0-cells of $\sC_{\cat V}$ are small $\cat V$-categories such that all hom-objects are cofibrant in $\cat V$.
    \item For any $\cat V$-categories $\cat A$ and $\cat B$, the category $\sC_\cat V(\cat A,\cat B)$ is $\Ho (\mab {\cat {Mod}}{\cat A}{\cat B})$, the homotopy category of the category of $(\cat A,\cat B)$-bimodules, in the explicit form of the category with bifibrant objects of $\mab {\cat {Mod}}{\cat A}{\cat B}$ as objects and homotopy classes of morphisms in $\mab {\cat {Mod}}{\cat A}{\cat B}$ as morphisms.
    \item Given $\mab M{\cat A}{\cat B}$ and $\mab N{\cat B}{\cat C}$, their composite $M\odot N$ is defined to be their derived tensor product over $\cat B$.  More explicitly, we set
    \[
    M\odot N = \hocolim_{\Delta^{\mathrm{op}}} \big(\Bar(M;\widehat{\cat B};N)\big)^f,
    \]
    the functorial fibrant replacement of the homotopy colimit of the simplicial bar construction.
\end{itemize}
\end{defn}

As in the case of $\sR_{\cat V}$, it is straightforward, if tedious, to check that $\sC_{\cat V}$ is indeed a bicategory. Note that for any 0-cell $\cat A$ in $\sC_{\cat V}$, the unit object $U_{\cat A}$ in $\sC_{\cat V}(\cat A,\cat A)$ is $\widehat{\cat A}$. Assumptions \ref{assumption:oneobj} and \ref{assumption:manyobj} and the cofibrancy assumption on the 0-cells of $\sC_{\cat V}$ together imply that horizontal composition is well defined, as in the one-object case.

\begin{rmk}\label{rmk:cofibVcat}
By \cite[Proposition 6.3]{schwede-shipley:equivmonmodcat}, under our hypotheses on $\cat V$, for every $\cat V$-category $\cat C$, there exists a $\cat V$-category $\cat {QC}$ with same object set such that all hom-objects are cofibrant in $\cat V$ and an acyclic fibration of $\cat V$-categories, $\cat {QC} \to \cat C$, that is the identity on objects and a weak equivalence on hom-objects. The cofibrancy hypothesis on the 0-cells of $\sC_{\cat V}$ is therefore not too restrictive. 
\end{rmk}

The following computation of a particular derived tensor product proves useful to us below.

\begin{lem}\label{lem:bar-compute} Let $\cat A$, $\cat B$, and $\cat C$ be $\cat V$-categories.  For every pair of $\cat V$-functors $F:\cat A \to \cat C$ and $G:\cat B \to \cat C$ and every $a\otimes b\in \Ob (\cat A^{\mathrm op} \otimes \cat B)$, there is a natural isomorphism
$$|\Bar(\mab {\cat C}F{\Id};\widehat{\cat C};\mab {\cat C}{\Id}G)|(a\otimes b)\cong \mab {\cat C}FG(a\otimes b)$$
in $\Ho \cat {V}$.
\end{lem}

\begin{proof} Recall that the $n^{\text{th}}$-level of the bar construction is 
$$\operatorname{Bar}_n (\mab {\cat C}F{\Id};\widehat{\cat C};\mab {\cat C}{\Id}G)= \mab {\cat C}F{\Id} \boxtimes \widehat{\cat C}^{\boxtimes n}\boxtimes \mab {\cat C}{\Id}G.$$
Evaluating this functor at $a\otimes b\in \Ob (\cat A^{\mathrm op} \otimes \cat B)$, we obtain
$$\big(\mab {\cat C}F{\Id} \boxtimes \widehat{\cat C}^{\boxtimes n}\boxtimes \mab {\cat C}{\Id}G\big)(a\otimes b)= \coprod_{c_0,..., c_n}\cat C\big(F(a), c_0\big)\otimes \cat C(c_0,c_1)\otimes \cdots \otimes \cat C(c_{n-1}, c_n)\otimes \cat C\big(c_n, G(b)\big).$$   

The simplicial object $\Bar(\mab {\cat C}F{\Id};\widehat{\cat C};\mab {\cat C}{\Id}G)$ admits an augmentation
$$\coprod_{c_0}\cat C\big(F(a), c_0\big)\otimes \cat C\big(c_0, G(b)\big) \to \cat C\big( F(a), G(b)\big)$$
given by the composition in $\cat C$.  This augmented simplicial object admits extra degeneracies, defined from level $-1$ to level 0 by
$$ \cat C\big( F(a), G(b)\big)\cong S\otimes \cat C\big( F(a), G(b)\big) \to \cat C\big( F(a), F(a)\big) \otimes \cat C\big( F(a), G(b)\big)  \hookrightarrow \coprod_{c_0}\cat C\big(F(a), c_0\big)\otimes \cat C\big(c_0, G(b)\big), $$
where the first morphism is given by the unit of $F(a)$, and analogously in higher levels.  (Note that we could instead have chosen to define the extra degeneracies in terms of the unit on $G(b)$.)

It follows that the augmentation map induces a natural weak equivalence 
$$|\Bar(\mab {\cat C}F{\Id};\widehat{\cat C};\mab {\cat C}{\Id}G)|(a\otimes b)\xrightarrow{\simeq} \mab {\cat C}FG(a\otimes b)$$
in $\cat V$ and thus a natural isomorphism in $\Ho \cat V$. 
\end{proof}

\begin{rmk}
There is a faithful bifunctor $$\Theta\colon \sR_{\cat V} \to \sC_{\cat V}$$ defined as follows.  A 0-cell $A$ of $\sR_{\cat V}$, i.e., a monoid in $\cat V$ with cofibrant underlying object in $\cat V$, is sent to the $\cat V$-category $\Theta(A)$ with one object $*$, $\Theta(A)(*,*)=A$, and composition given by the multiplication in $A$.  If $\mab MAB$ is a 1-cell of $\sR_{\cat V}$ with left $A$-action $\lambda$ and right $B$-action $\rho$, then $\Theta(M)$ is the $\big(\Theta(A), \Theta(B)\big)$-bimodule
\[
\Theta (M): \Theta(A)^{\mathrm{op}}\otimes \Theta(B) \to \cat V
\]
that sends the unique object $* \otimes *$ of the source to $M$, and that is defined on the $\cat V$-object of morphisms 
to be the transpose
\[
A^{\mathrm{op}}\otimes B \to \cat V (M,M)
\]
of the bimodule action
\[
M\otimes A^{\mathrm{op}}\otimes B\cong A\otimes M \otimes B \xrightarrow{\rho(\lambda (-\otimes -)\otimes -)} M.
\]
We can extend $\Theta$ to 2-cells in the obvious way.
\end{rmk}

\begin{rmk} Let $\cat {Perf}_A$ denote the full sub-$\cat V$-category of \emph{perfect} left $A$-modules, i.e., those cofibrant $A$-modules $M$ that are finitely generated and such that 
\[
\hom _A(M,A)\otimes_A ^{\mathbb L} N \cong \hom _A(M,N)
\]
in $\Ho \cat V$ for every left $A$-module $N$, where  
$-\otimes^{\mathbb L}-$ is the total left derived functor of $$-\otimes_A -: \cat {Mod}_A\times {}_A\cat {Mod}\to \cat V.$$ 
Note that the finite generation hypothesis implies that $\cat {Perf}_A$ is small. 

For every $A\in \Ob \cat {Mon}(\cat V)$ with cofibrant underlying object, there is a faithful $\cat V$-functor $$\iota_A\colon \Theta(A) \to \cat {Perf}_A$$ that sends the unique object to $A$ and that is defined on the $\cat V$-object of morphisms to be the transpose
\[
A \to \cat{Perf}_A(A,A)
\]
of the multiplication map $\mu\colon A\otimes A \to A$.

Since the 0-cells of the bicategory $\sC_{\cat V}$ are required to satisfy a cofibrancy condition, we apply cofibrant replacement and consider $\cat {QPerf}_A$ instead (cf. Remark \ref{rmk:cofibVcat}), which is small, since $\cat{Perf}_A$ is, and  therefore a $0$-cell of $\sC_{\cat V}$. Since $A$ and therefore $\Theta(A)$ are cofibrant, the $\cat V$-functor $\iota_A$ lifts through the acylic fibration $\cat{QPerf}_A\to \cat {Perf}_A$, giving rise to a 
$\cat V$-functor $$\iota_A'\colon \Theta(A) \to \cat {QPerf}_A$$

As special cases of Example \ref{exmp:bimodule}, the functor $\iota_A'$ gives rise to two $\cat V$-category bimodules, 
$$\mab {(\cat{QPerf}_A)} {\iota _A'}{\Id} \in \Ob \mab {\cat {Mod}}{\Theta(A)}{\cat{QPerf}_A}$$ 
and
$$\mab {(\cat{QPerf}_A)}{\Id} {\iota' _A}\in \Ob \mab {\cat {Mod}}{\cat{QPerf}_A}{\Theta(A)}.$$
\end{rmk}

\begin{prop}\label{prop:perfmorita}
For every $A\in \Ob \cat {Mon}(\cat V)$ with cofibrant underlying object in $\cat V$, the pair 
\[
\Big(\,\mab {(\cat{QPerf}_A)} {\iota' _A}{\Id},\, \mab {(\cat{QPerf}_A)}{\Id} {\iota' _A}\,\Big)
\]
is a Morita equivalence in $\sC_{\cat V}$ between $\Theta(A)$ and $\cat{QPerf}_A$.
\end{prop}

\begin{proof}
The definition in Example \ref{exmp:bimodule} says that 
\[
\mab {(\cat{QPerf}_A)} {\iota' _A}{\Id}\colon \Theta(A)^{\mathrm{op}}\otimes \cat {QPerf}_A \to \cat V\colon (*,M)\mapsto \cat{QPerf}_A(A,M),
\]
while 
\[
\mab {(\cat{QPerf}_A)} {\Id}{\iota' _A}\colon \cat {QPerf}_A^{\mathrm{op}}\otimes \Theta(A)  \to \cat V\colon (M,*)\mapsto \cat{QPerf}_A(M,A).
\]
We show first that $\mab {(\cat{QPerf}_A)} {\iota' _A}{\Id}\odot \mab {(\cat{QPerf}_A)}{\Id} {\iota' _A} \cong U_{\Theta(A)}$ and $\mab {(\cat{QPerf}_A)}{\Id} {\iota' _A}\odot \mab {(\cat{QPerf}_A)}{\iota' _A}{\Id}\cong U_{\cat{QPerf}_A}$.

Observe that 
\begin{align*}
  \mab {(\cat{QPerf}_A)} {\iota' _A}{\Id}\odot \mab {(\cat{QPerf}_A)}{\Id} {\iota' _A}&=|\Bar(\mab {(\cat{QPerf}_A)} {\iota' _A}{\Id}; \widehat {\cat{QPerf}_A};\mab {(\cat{QPerf}_A)}{\Id} {\iota' _A})|\\
  &\cong \mab {(\cat{QPerf}_A)} {\iota' _A}{\iota' _A}\\
  &\cong \mab {(\cat{Perf}_A)} {\iota _A}{\iota _A}\\
  &\cong \mab {\Theta (A)} {\Id}{\Id}\\
  &=\widehat{\Theta (A)}=U_{\Theta(A)},
\end{align*}
where the first isomorphism follows from Lemma \ref{lem:bar-compute} and the second from the fact that $\cat{QPerf}_A$ is weakly equivalent to $\cat{Perf}_A,$ while the third is due to the fact that 
$$\Theta(A)(*,*) =A \cong \cat {Perf}_A(A,A).$$

On the other hand,
$$ \mab {(\cat{QPerf}_A)}{\Id} {\iota' _A}\odot \mab {(\cat{QPerf}_A)}{\iota' _A}{\Id} =|\Bar\big(\mab{(\cat{QPerf}_A)}{\Id} {\iota' _A}; \widehat{\Theta (A)};\mab {(\cat{QPerf}_A)}{\iota' _A}{\Id}\big)|.$$
For all $n\geq 0$, the $\cat V$-functor
$$\operatorname{Bar}_n \big(\mab{(\cat{QPerf}_A)}{\Id} {\iota' _A}; \widehat{\Theta (A)};\mab {(\cat{QPerf}_A)}{\iota' _A}{\Id}\big)= \mab{(\cat{QPerf}_A)}{\Id} {\iota' _A} \boxtimes \widehat{\Theta (A)}^{\boxtimes n}\boxtimes \mab {(\cat{QPerf}_A)}{\iota' _A}{\Id}$$
from $\cat{QPerf}_A^{\mathrm op}\otimes \cat{QPerf}_A$ to  $\cat V$ is specified on a pair of perfect $A$-modules $M$ and $N$ by
\begin{align*}
\operatorname{Bar}_n \big(\mab{(\cat{QPerf}_A)}{\Id} {\iota' _A}; \widehat{\Theta (A)};&\mab {(\cat{QPerf}_A)}{\iota' _A}{\Id}\big)(M\otimes N)\\
&=\mab{(\cat{QPerf}_A)}{\Id} {\iota' _A} (M\otimes \iota'_A(*)) \otimes A^{\otimes n}\otimes \mab {(\cat{QPerf}_A)}{\iota' _A}{\Id}(\iota'_A(*)\otimes N)\\
&=\cat{QPerf}_A(M,A)\otimes A^{\otimes n}\otimes  \cat{QPerf}_A(A,N)\\
&\cong \cat{QPerf}_A(M,A)\otimes A^{\otimes n}\otimes  N,
\end{align*}
where the last isomorphism follows from the facts that $\cat V$ is cofibrantly generated and that there is a weak equivalence $\cat{QPerf}_A(A,N) \to \cat{Perf}_A(A,N)\cong N$, with cofibrant domain and codomain.

It follows that for all perfect $A$-modules $M$ and $N$
\begin{align*} 
 \Big(\mab {(\cat{QPerf}_A)}{\Id} {\iota' _A}\odot \mab {(\cat{QPerf}_A)}{\iota' _A}{\Id}\Big)(M\otimes N)&\cong |\Bar\big(\cat {QPerf}_A(M,A); A; N\big)|\\
 &= \cat {QPerf}_A(M,A) {\otimes_A}^{\mathbb L} N\\
  &\cong \cat {QPerf}_A(M,N)\\
  &=  \mab {(\cat{QPerf}_A)}{\Id} {\Id}(M\otimes N)\\
  &= \widehat {\cat{QPerf}_A}(M\otimes N) =U_{\cat{QPerf}_A}(M\otimes N)
\end{align*}
where the second isomorphism holds since $\cat {QPerf}_A(M,M')$ and $\cat {Perf}_A(M,M')=\hom_A(M,M')$ are weakly equivalent for every $A$-module $M'$, and $M$ is perfect. Thus 
$$\mab {(\cat{QPerf}_A)}{\Id} {\iota' _A}\odot \mab {(\cat{QPerf}_A)}{\iota' _A}{\Id}\cong U_{\cat{QPerf}_A},$$
as desired.

The natural isomorphisms above provide us with the required structure maps for the dual pairs $\Big(\mab {(\cat{QPerf}_A)}{\Id} {\iota' _A},  \mab {(\cat{QPerf}_A)}{\iota' _A}{\Id}\Big)$ and $\Big(\mab {(\cat{QPerf}_A)}{\iota' _A}{\Id} ,  \mab {(\cat{QPerf}_A)}{\Id}{\iota' _A}\Big)$.  It remains only to show that triangle equalities are satisfied.

Checking the triangle inequalities is a straightforward exercise, if one starts by describing $$\mab {(\cat{QPerf}_A)}{\iota' _A}{\Id}\odot \mab {(\cat{QPerf}_A)}{\Id} {\iota' _A}\odot \mab {(\cat{QPerf}_A)}{\iota' _A}{\Id}$$ and $$\mab {(\cat{QPerf}_A)}{\Id} {\iota' _A}\odot \mab {(\cat{QPerf}_A)}{\iota' _A}{\Id}\odot \mab {(\cat{QPerf}_A)}{\Id} {\iota' _A}$$ as realizations of bisimplicial objects and $\mab {(\cat{QPerf}_A)}{\iota' _A}{\Id}$ and $\mab {(\cat{QPerf}_A)}{\Id} {\iota' _A}$ as realizations of constant bisimplicial objects. The evaluations and coevaluations can then be easily described in terms of bisimplicial maps that clearly satisfy the triangle inequalities.
\end{proof}

\begin{defn}
Let $\cat A$ be a $\cat V$-category, and let $M$ be an $\cat A$-bimodule. 
The \emph{Hochschild construction} on $\cat A$ with coefficients in $M$, denoted $\HH_{\cat{Vcat}}(\cat A;M)$, is $|B_{\bullet}^{\mathrm {cyc}}(\cat A;M)|$, the geometric realization of the cyclic bar construction on $\cat A$ with coefficients in $M$.
\end{defn}

\begin{rmk}
Since the cyclic bar construction can be viewed as a simplicial object in $\cat V$ (cf. Definition \ref{defn:bar-multiobj}), the Hochschild construction on $\cat A$ with coefficients in $M$ can be viewed as an object in $\cat V$.
\end{rmk}

\begin{prop}\label{prop:Hochschild-manyobj}
The Hochschild construction $\HH_{\cat {VCat}}$ extends to a $\Ho \cat {V}$-shadow on $\sC_{\cat V}$. 
\end{prop}

\begin{proof}
The Hochschild construction clearly defines a functor 
\[
\HH_{\cat {VCat}} (\cat A; -)\colon \Ho (\mab{\cat {Mod}}{\cat A}{\cat A}) \to \Ho \cat V
\]
for every $\cat V$-category $\cat A$. As in the proof of Proposition \ref{prop:HH_is_shadow},  the usual Dennis--Morita--Waldhausen argument \cite[Proposition 6.2]{blumberg-mandell:localization} shows that 
\[
\HH_{\cat {VCat}} (\cat A;M\odot N) \cong \HH_{\cat {VCat}} (\cat B;N\odot M)
\]
naturally, for all $\mab M{\cat A}{\cat B}$ and $\mab N{\cat B}{\cat A}$. The compatibility of the natural isomorphism above with the associator and the unitor again follows immediately from the associativity and unit coherence of the symmetric monoidal structure on $\cat V$. 
\end{proof}

As in the single-object case, the proposition above and Proposition \ref{prop:morita} together imply that Morita invariance holds in the many-object case as well.

\begin{cor}\label{cor:morita-many} Let $(\mab M{\cat A}{\cat B}, \mab N{\cat B}{\cat A})$ be a Morita equivalence in $\sC_{\cat V}$.  For any $\cat A$-bimodule $Q$, there is an isomorphism
\[
\HH_{\cat {VCat}}(\cat A;Q) \xrightarrow{\cong} \HH _{\cat {VCat}}(\cat B;N\odot Q\odot M).
\]
In particular, the Euler characteristic
\[
\chi(M)\colon \HH_{\cat {VCat}}(\cat A;\cat A) \to \HH_{\cat {VCat}}(\cat B;\cat B)
\]
is an isomorphism.
\end{cor}

The Hochschild construction on $\cat{VCat}$ extends that on $\cat V$, as mediated by the bifunctor $\Theta$.

\begin{lem}\label{lem:v-to-vcat}
For any monoid $A$ in $\cat V$ and any $A$-bimodule $M$,
\[
\Theta\big(\HH_{\cat V}(A;M)\big)\cong \HH_{\cat {VCat}}\big(\Theta(A);\Theta(M) \big).
\]
\end{lem}

\begin{proof}
Observe that 
$$\Theta\big(\HH_{\cat V}(A;M)\big): \Theta (S)^{\mathrm {op}}\otimes \Theta (S) \to \cat V$$
is the functor that sends the unique object $*\otimes *$ to $\HH_{\cat V}(A;M)$ and that is defined on the morphism object to be the unit of the endomorphism object for $\HH_{\cat V}(A;M)$,
$$ S\otimes S \cong S \to \cat V \big(\HH_{\cat V}(A;M), \HH_{\cat V}(A;M)\big).$$

On the other hand, the $n^{\text{th}}$ level of the cyclic bar construction $B_\bullet ^{\mathrm {cyc}}\big(\Theta (A);\Theta (M) \big)$ is the restriction of the functor
$$\Theta(M) \boxtimes \Theta (A)^{\boxtimes n}: \Theta (A)^{\mathrm{op}}\otimes \Theta(A) \to \cat V,$$
which sends the unique object $*\otimes *$ to $M\otimes A^{\otimes n}$, to  $\Theta (S)^{\mathrm {op}}\otimes \Theta (S)$.  (This restriction is necessary, since the face maps of the cyclic bar construction are not $\Theta(A)$-bimodule maps.)
In particular $\Theta(M) \boxtimes \Theta (A)^{\boxtimes n}|_{\Theta (S)^{\mathrm {op}}\otimes \Theta (S)}$ is defined on the morphism object $S$ to be the unit of the endomorphism object on $M\otimes A^{\otimes n}$,
$$ S\otimes S \cong S \to \cat V \big(M\otimes A^{\otimes n}, M\otimes A^{\otimes n}\big).$$
The geometric realization of $B_\bullet ^{\mathrm {cyc}}\big( \Theta (A);\Theta (M)  \big)$ is thus clearly isomorphic to $\Theta\big(\HH_{\cat V}(A;M)\big)$.
\end{proof}

An ``agreement"-type result now follows from Propositions \ref{prop:morita}, \ref{prop:perfmorita}, and \ref{prop:Hochschild-manyobj}.

\begin{cor}\label{cor:agreement} For any monoid $A$ in $\cat V$,
\[
\Theta\big(\HH_{\cat V}(A;\widehat{A})\big)\cong  \HH_{\cat{VCat}}\big(\,\cat {QPerf}_A;\widehat{\cat {QPerf}_A}\big).
\]
\end{cor}

 \begin{exmp}

 The theory of unstable topological Hochschild homology fits into this framework. Introduced by Nikolaus \cite{Arb18}, unstable THH is the Hochschild homology of small categories enriched in Kan complexes. In particular, $u\THH =  \HH_{\cat{VCat}}$ for $\cat V$ the category of Kan complexes, which satisfies Assumption \ref{assumption:manyobj}. As this fits into the general framework developed above, by Proposition \ref{prop:Hochschild-manyobj}  unstable topological Hochschild homology is a shadow. It is therefore Morita invariant.

 \end{exmp}
 
 A possible analogue of Proposition~\ref{prop:Hochschild-manyobj} is given by replacing simplicial categories with differential graded categories.
 
\begin{rmk} 
 In \cite[Proposition 3.11, Theorem 4.7]{schweigert-woike}, Schweigert and Woike show that the differential graded conformal block for the torus is an example of the Hochschild complex for a differential graded category. Although we will not pursue a proof in this paper, we suggest that these dg Hochschild constructions are also examples of bicategorical shadows.  An alternative to showing this directly would be to pass back and forth between the differential graded and simplicial contexts using the zig zag of Quillen equivalences in \cite[Section 6]{tabuada}.

\end{rmk}

\section{Application to equivariant Hochschild theories}\label{sec:equivHochschild}

One key motivation for developing a general bicategorical framework for Hochschild homology is to study equivariant versions of Hochschild homology. In particular, in \cite{AnBlGeHiLaMa}, Angeltveit, Blumberg, Gerhardt, Hill, Lawson, and Mandell define a theory of $C_n$-twisted topological Hochschild homology for a $C_n$-equivariant ring spectrum, $\THH_{C_n}(R)$. This theory has an algebraic analogue, the twisted Hochschild homology of Green functors, developed in \cite{BlGeHiLa}. In this section we set up a framework for Hochschild theories twisted by an automorphism and use this framework to show that both twisted THH and Hochschild homology for Green functors are Morita invariant. 
\subsection{Hochschild theories twisted by an automorphism}\label{Section:twisted}
Let $(\cat V, \otimes, S)$ be a symmetric monoidal, simplicial model category satisfying Assumption \ref{assumption:oneobj}.  Generalizing the contexts of both $C_n$-twisted topological Hochschild homology of $C_n$-spectra  and $C_n$-twisted Hochschild homology of $C_n$-Green functors, we consider the category $\Vaut$ of objects in $\cat V$ equipped with automorphisms. More precisely, the objects of $\Vaut$ are pairs $(X,\varphi)$, where $X$ is an object of $\cat V$, and $\varphi\colon X \to X$ is an automorphism; morphisms in $\Vaut$ are morphisms in $\cat V$ that commute with the automorphism.  The monoidal structure on $\cat V$ clearly lifts to $\Vaut$, which also naturally inherits a simplicial structure from $\cat V$. We denote the lifted tensor product on $\Vaut$ by $\otimes$, and note that the unit object is $(S, \Id_S)$.

\begin{exmps}
\label{exmp:Vaut}
Let $C_n$ denote the cyclic group of order $n$, and let $g\in C_n$ be a generator. The following two examples of $\Vaut$ are the ground categories for $C_n$-twisted THH and $C_n$-twisted Hochschild homology for Green functors, respectively.
\begin{enumerate}
	\item Let $\cat V$ be the symmetric monoidal category of orthogonal $C_n$-spectra, denoted by $(\cat{Sp}^{C_n}, \wedge, S^0).$ By \cite[Chapter III, Theorem 4.2]{mandell-may} the stable model structure on orthogonal $G$-spectra is cofibrantly generated and $G$-topological and therefore topological, whence also a simplicial model structure by \cite[Section 3.8]{riehl}.  Moreover, the desired model structures on categories of bimodules exist by \cite[Chapter III, Theorem 7.6]{mandell-may}, so that Assumption \ref{assumption:oneobj} is indeed satisfied.
	Since $C_n$ is abelian, the conjugation map $c_g$ is the identity. 
Thus for any $C_n$-spectrum $X$, left multiplication by $g$ induces an automorphism,
\[
	g: X=c_g^* X \to X: x\mapsto gx.
\]
This defines an object $(X,g)$ in $\cat{Sp}^{C_n}_\cat{aut}.$
\item The category of simplicial $C_n$-Mackey functors, denoted by $s\mathsf{Mack}_{C_n}$, is a symmetric monoidal simplicial model category by \cite[Theorem 4.3]{BlGeHiLa}. This category is cofibrantly generated, similarly to the category of simplicial abelian groups. (For $G = \{e\}$, $s\Mack = s\cat{Ab}$.) By Proposition 4.4 of \cite{BlGeHiLa}, every object in this category is fibrant. Therefore by Lemma A.3 of \cite{schwede-shipley}, $s\mathsf{Mack}_{C_n}$ satisfies Assumption \ref{assumption:oneobj}.
For each $C_n$-Mackey functor $\m M$ and subgroup $H\subset C_n$, the Weyl group $W_H(C_n)=C_n/H$ acts on $\m M(C_n/H)$, the $C_n$-Mackey functor $\m M$ evaluated at the orbit $C_n/H$. Therefore we can define
a levelwise $g$-action on a $C_n$-Mackey functor $\m M$ by passing to the Weyl group, i.e.,
\[
g(C_n/H): \m{M}(C_n/H)\xrightarrow{[g]\in C_n/H} \m{M}(C_n/H), \forall H\subset C_n.
\]
This levelwise action is compatible with transfers and restrictions. Therefore it assembles into a $g$-action on the $C_n$-Mackey functor, as well as on a simplicial $C_n$-Mackey functor. Thus for a simplicial Mackey functor $\m{M}_\bullet$,  this defines an object $(\m {M}_\bullet,g)$ in $s\mathsf{Mack}_{C_n, \cat{aut}}.$
\end{enumerate}

In both cases, an object in $\cat V$ gives rise to an object in $\Vaut,$ where the automorphism is determined by $g.$
\end{exmps}

Note that the choice of an automorphism of an object $X$ in $\cat V$ is equivalent to the choice of an action of the infinite cyclic group $C_\infty$ on $X$.  More formally,  $\Vaut$ is isomorphic to  $\cat V^{\Sigma C_\infty}$, the category of functors into $\cat V$ from the one-object category $\Sigma C_\infty$ whose morphism set is $C_\infty$.  This identification enables us to prove easily that $\Vaut$ also inherits model structure from $\cat V$, as follows.

The inclusion of the trivial group into $C_\infty$ induces two pairs of adjoint functors
\[
\xymatrix{\cat V\ar @<1ex>[rr]^{}\ar @<-1ex>[rr]^{}&&\Vaut, \ar [ll]^{}}
\]
where the functor from $\Vaut$ to $\cat V$ forgets the automorphism. If, for example, $\cat V$ is a combinatorial model category (i.e., cofibrantly generated and locally presented), then by \cite[Theorem 3.4.1]{hkrs}, \cite[Theorem 2.23]{bhkkrs}, and \cite[Theorem 11.6.1]{hirschhorn}, $\Vaut$ admits two combinatorial model structures for which the weak equivalences are created in $\cat V$: the \emph{left-induced (or injective)} model structure, denoted $\Vautinj$, in which the cofibrations are created in $\cat V$  and the \emph{right-induced (or projective)} model structure, denoted $\Vautproj$, in which the fibrations are created in $\cat V$.  

By \cite[Lemma 2.24, 4.21, 4.22]{bhkkrs}, when equipped with the lifted monoidal structure, $\Vautinj$ is a monoidal and simplicial model category, which is left (respectively, right) proper if $\cat V$ is.  The analogous results also hold for $\Vautproj$ by \cite[Theorem 11.7.3]{hirschhorn}, which enables us to conclude that $\Vautproj$ is a simplicial model category.  To show that  $\Vautproj$ is also a monoidal model category, one can first observe that the $\cat V$-object underlying the internal hom in $\Vaut$ endows the free-forgetful adjunction between $\cat V$ and $\Vaut$ with the structure of $\cat V$-enriched adjunction, then use the characterization of monoidal model categories in terms of (acyclic) fibrations \cite[Lemma 4.2.2]{hovey}.

It follows that all of the constructions and results above can be applied to $\big(\Vaut, \otimes,(S, \Id_S)\big)$ equipped with its injective or projective model structure, giving rise to a $\Ho (\Vaut)$-shadow, $\HH_{\Vaut}$, on $\sR_{\Vaut}$ with the properties above, such as Morita invariance (Corollary \ref{cor:morita}).  Agreement (Corollary \ref{cor:agreement}) holds as well, at least as long as $\Vaut$ is a locally presentable base, and we restrict to $\Vaut$-categories with cofibrant morphism objects.

\begin{rmk} 
\label{rmk:twisted_bimodule}
The categories $\cat {Mon}(\Vaut)$ of monoids in $\Vaut$ and $\cat {Mon}(\cat V)_{\textrm{aut}}$ of monoids in $\cat V$ equipped with monoid automorphisms are clearly isomorphic.
For any pair $(A,\vp)$ and $(B, \psi)$ of monoids  in $\Vaut$, we consider the category 
$$\mab {\cat {Mod}(\Vaut)}{(A,\vp)}{(B,\psi)}$$
of $\big((A,\vp)$,$(B,\psi)\big)$-bimodules in $\Vaut.$

It is straightforward to check that this category has the following explicit description.

An object of $\mab {\cat {Mod}(\Vaut)}{(A,\vp)}{(B,\psi)}$ consists of the following data.
\begin{itemize}
\item 
morphisms in $\cat V$:
$\lambda \colon A \otimes M \to M$, $\rho\colon M\otimes B \to M$, and $\gamma\colon M\to M$  such that
    \item $(M, \lambda, \rho)\in \Ob\mab {\cat {Mod}}AB$,
    \item $\gamma$ is an isomorphism, and
    \item the diagram
    \[ \xymatrix{A\otimes M \ar[d]_{\vp\otimes \gamma}\ar [r]^\lambda & M\ar [d]^\gamma&M\otimes B \ar[l]_\rho\ar[d]^{\gamma\otimes \psi}\\
        A\otimes M \ar [r]^\lambda & M&M\otimes B \ar[l]_\rho}
    \]
    commutes.
\end{itemize}
 The notation we use for an object in this category is $(\mab M\vp \psi, \gamma)$. Note that by definition, $\gamma$ can be viewed as an isomorphism of $(A,B)$-bimodules $\gamma: M^{\psi^{-1}}\xrightarrow{\cong} {^\vp\!M}$.
\end{rmk}

\begin{defn}
\label{def:phi-twisted_HH}
Let $(A, \vp)$ be a monoid in $\Vaut$, and let $(\mab M\vp \vp, \gamma)$ be an $(A, \vp)$-bimodule with left action $\lambda: A\otimes M\to M$ and right action $\rho: M\otimes A\to M$.
The \emph{$\vp$-twisted Hochschild homology} of $A$ with coefficients in $(\mab M\vp \vp, \gamma)$ is defined to be
\[
\HH ^\vp(A; (\mab M\vp \vp, \gamma)) =\HH_\cat V (A;{^\vp\!M}),
\]
where $^\vp\!M$ denotes the $A$-bimodule with left action $\lambda\circ(\vp\otimes\id): A\otimes M\to M$ and right action $\rho.$
\end{defn}

 Let $(A, \vp)$ be a monoid in $\Vaut$.
 The unit element $U(A,\vp)$ in $(A,\vp)$-bimodules is $(\mab A \mu \mu,\vp)$. We define the \emph{$\vp$-twisted Hochschild homology} of $A$ to be the $\vp$-twisted Hochschild homology of $A$ with these coefficients, denoted by
$\HH ^\vp(A)$.

\begin{exmp}
	We will see in Section \ref{sec:TwistedTHH} that when $A$ is a $C_n$-ring spectrum, equipped with the automorphism $\vp$ given by the action of a generator of $C_n$, the $\vp$-twisted Hochschild homology of $A$ is exactly twisted topological Hochschild homology (see Proposition \ref{prop:THH_case}). In Section \ref{subsec:HH_case} we will show that twisted Hochschild homology for Green functors also fits into this framework. 

\end{exmp}

We now show that Morita invariance holds for $\varphi$-twisted Hochschild homology.

\begin{prop}\label{cor:twisted_morita} Let $(A,\vp),(B,\psi)$ be two objects in $\cat{Mon}(\Vaut).$
Suppose the dual pair \\$\big((\mab M {\vp} {\psi},\gamma_M), (\mab N{\psi}{\vp},\gamma_N)\big)$ is a Morita equivalence in $\sR_{\Vaut}$.  For any $(A,\vp)$-bimodule $Q$, there is an isomorphism
\[
\HH^{\vp}(A;Q) \xrightarrow{\cong} \HH^{\psi}(B;N\odot Q\odot M).
\]
In particular, for $Q = (\mab A \mu \mu, \vp)$ the Euler characteristic
\[
\chi(M)\colon \HH^{\vp}(A;A) \to \HH^{\psi}(B;B)
\]
is an isomorphism.
\end{prop}

\begin{proof}
It is clear that  {$\vp$-twisted Hochschild homology} extends to a functor 
\[
 \HH ^\vp(A; -) : \Ho\big(\mab {\cat {Mod}(\Vaut)}{(A,\vp)}{(A,\vp)}\big)\to \Ho \cat V
\] for each $(A,\vp)\in \cat{Mon}(\Vaut).$
	
	We start by proving cyclic invariance. For any $(A,\vp),(B,\psi)\in \cat{Mon}(\Vaut)$, let $(_\vp M_\psi, \gamma_M)$ be an $\big((A,\vp)\text,(B,\psi)\big)$-bimodule and $(_\psi N_\vp, \gamma_N)$  a $\big((B,\psi)\text,(A,\vp)\big)$-bimodule. As mentioned in Remark \ref{rmk:twisted_bimodule}, $\gamma_N$ can also be seen as an isomorphism from $N^{\vp^{-1}}$ to ${^\psi\!N}$.
	
	Observe that	
\begin{align*}
	\HH ^\vp(A; M\odot N)&= \HH_\cat V (A,^\vp\!\lvert \Bar(M;B;N)\lvert )\\
	&=\lvert ^\vp\!M\otimes B^{\otimes\bullet} \otimes N\otimes A^{\otimes\bullet}\lvert\\
	&\cong \lvert N\otimes A^{\otimes\bullet} \otimes {^\vp}\!M\otimes B^{\otimes\bullet}\lvert\\
	&\cong \lvert N^{\vp^{-1}}\otimes A^{\otimes\bullet} \otimes M\otimes B^{\otimes\bullet}\lvert\\
	&\cong \lvert ^\psi\!N\otimes A^{\otimes\bullet} \otimes M\otimes B^{\otimes\bullet}\lvert=\HH ^\psi(B; N\odot M),
\end{align*}
where the first isomorphism follows from the Dennis-Morita-Waldhausen argument, the second from the graded isomorphism
$$N\otimes \vp ^{\otimes \bullet}\otimes M : N\otimes A^{\otimes\bullet} \otimes {^\vp}\!M\otimes  \xrightarrow{\cong} N^{\vp^{-1}}\otimes A^{\otimes\bullet} \otimes M,$$
 which can easily be shown to commute with face and degeneracy maps, and the third  from the isomorphism $\gamma_N: N^{\vp^{-1}}\xrightarrow{\cong}{^\psi\!N}.$
 
 Having established cyclic invariance, we consider the composite 
 \[
 \HH^{\vp}(A;Q) \xrightarrow{\cong} \HH^{\vp}(A;Q \odot U_A) \xrightarrow{\cong}\HH^{\vp}(A;Q \odot M \odot N) \xrightarrow{\cong} \HH^{\psi}(B;N\odot Q\odot M).
 \]
 The first two isomorphisms are from the definition of a Morita equivalence, and the third isomorphism is cyclic invariance. This gives the desired Morita invariance statement.
\end{proof}

\begin{rmk}\label{rmk:notashadow}
We observe that although the proof of Proposition  \ref{cor:twisted_morita} shows that $\vp$-twisted Hochschild homology satisfies cyclic invariance, it is \emph{not} a shadow. Indeed, the diagram in Definition \ref{defn:shadow} relating the cyclic invariance map, $\theta$, and the associator, $a$, does not commute. However, as shown above, cyclic invariance is sufficient to prove Morita invariance.
\end{rmk}

\subsection{Twisted topological Hochschild homology}\label{sec:TwistedTHH}
A key example that fits into the framework described above is twisted topological Hochschild homology. In \cite{AnBlGeHiLaMa}, Angeltveit, Blumberg, Gerhardt, Hill, Lawson, and Mandell defined a generalization of topological Hochschild homology for $C_n$-equivariant ring spectra. In this section we use the framework developed above to prove that twisted THH is Morita invariant. We begin by recalling the definition of $C_n$-twisted THH for $C_n$-ring spectra. 

The definition uses the $C_n$-twisted cyclic bar construction $B_{\bullet}^{\text{cyc}, C_n}$, which produces a simplicial spectrum defined as follows. Let $R$ be a $C_n$-ring spectrum, and let $M$ be an $R$-bimodule with left $R$-action $\lambda$ and right $R$-action $\rho$.
The $q$-th level of the simplicial object $B_{\bullet}^{\text{cyc}, C_n}(R;M)$ is
$$B_{q}^{\text{cyc}, C_n}(R;M) = M\wedge R^{\wedge q},$$
with the face and degeneracy maps  given by $$
d_i = \left\{\begin{array}{ll} 
\rho \wedge \Id^{\wedge (q-1)} &: i=0,\\
\Id^{\wedge i} \wedge \mu \wedge \Id^{\wedge (q-i-1)} &: 0< i<q \\  (\lambda \wedge \Id^{\wedge (q-1)}) \circ \alpha_q &: i=q,
\end{array}\right.
$$
and 
$$
s_i =  \Id^{\wedge (i+1)} \wedge \eta \wedge \Id^{\wedge (q-i)} \hspace{1cm}\forall\,  0\leq i \leq q,
$$
where $\mu$ is the multiplication map, $\eta$ is the unit map, and $\alpha_q$ is the map that cyclically permutes the last factor to the front and then acts on the new first factor by $g$. 

When $M=R$ with the canonical $R$-bimodule structure, we write $B_{\bullet}^{\text{cyc}, C_n}(R)$ for this twisted cyclic bar construction. In this case, the $C_n$-twisted cyclic bar construction admits the structure of a $\Lambda_n^{\op}$-object in the sense of B\"okstedt-Hsiang-Madsen \cite{BHM}, whence, as observed in \cite{BHM}, the geometric realization of the $C_n$-twisted cyclic bar construction of $R$, $|B_{\bullet}^{\text{cyc}, C_n}(R)|$, admits an $S^1$-action. In \cite{AnBlGeHiLaMa} the authors define $C_n$-twisted topological Hochschild homology using this twisted cyclic bar construction as follows.

\begin{defn}
Let $U$ be a complete $S^1$-universe and let $\widetilde{U} = \iota^*_{C_n} U$ be the pullback of the universe $U$ to $C_n$. Let $R$ be an associative orthogonal $C_n$-ring spectrum indexed on $\widetilde{U}.$ The $C_n$-twisted topological Hochschild homology of $R$ is defined to be the norm $N_{C_n}^{S^1}(R)$, which is given by 
\[
\THH_{C_n}(R) = \mathcal{I}_{\mathbb{R}^{\infty}}^U |B_{\bullet}^{\text{cyc}, C_n}(\mathcal{I}_{\widetilde{U}}^{\mathbb{R}^{\infty}} R)|,
\]
where $\mathcal I$ denotes a change-of-universe functor.
\end{defn}

In this paper, we extend the definition of $C_n$-twisted topological Hochschild homology to $C_n$-twisted topological Hochschild homology with coefficients. 

\begin{defn}
Let $R$ be an associative orthogonal $C_n$-ring spectrum, and let $M$ be an $R$-bimodule. The \emph{$C_n$-twisted topological Hochschild homology of $R$ with coefficients in $M$} is defined to be the following 
$$
\THH_{C_n}(R;M) = |B^{\text{cyc}, C_n}_{\bullet}( R;M)|.
$$
\end{defn}

Note that $C_n$-twisted topological Hochschild homology with coefficients is in general  not an $S^1$-spectrum, but rather a $C_n$-spectrum. Throughout this article, we are considering $C_n$-twisted topological Hochschild homology as a $C_n$-spectrum. This version of $C_n$-twisted topological Hochschild homology is an example of a twisted Hochschild homology theory as in Definition \ref{def:phi-twisted_HH}. In this case, let $(\cat V ,\otimes, S)$ be the category of orthogonal $C_n$-equivariant spectra.

\begin{prop}
\label{prop:THH_case}
 Let $R$ be an orthogonal $C_n$-ring spectrum, and let $g$ be a generator of $C_n$. The generator $g$ defines an automorphism of $R$, and there is a natural isomorphism between the $C_n$-twisted THH of $R$ (as a $C_n$-spectrum) and the $g$-twisted Hochschild homology of $R$:
	\[
\THH_{C_n}(R)\cong \HH ^g(R).
\]

\end{prop}

\begin{proof}

By definition, $\THH_{C_n}(R)$ is an $S^1$-spectrum. Its underlying $C_n$-spectrum is the geometric realization of the cyclic bar construction $B^{\text{cyc}}_{\bullet}(R, {}^g\!R )$, where ${}^g\!R$ denotes the $R$-bimodule $R$, with the left action twisted by $g: R\to R,$ as in Definition \ref{def:phi-twisted_HH}. It's easy to see that it is an isomorphism on the level of simplicial spectra.

 \end{proof}

Having established that twisted THH fits into the general framework developed in Section \ref{Section:twisted}, Morita invariance then follows immediately from Proposition \ref{cor:twisted_morita}; we consider the bicategory $\sR_{\cat {Sp}^{C_n}}$, a special case of the bicategory $\sR_{\cat V}$ in Definition \ref{def:bicat-V} with $\cat V = \cat{Sp}^{C_n}$, the category of orthogonal $C_n$-spectra (see Example \ref{exmp:Vaut}(1); in particular, Assumption \ref{assumption:oneobj} is satisfied).

We now unpack the notions of Morita equivalence and Morita invariance in this setting. It follows from Definition \ref{def:Morita} that $C_n$-ring spectra $R$ and $S$ are Morita equivalent if there exists an $(R,S)$ $C_n$-bimodule $M$ and an $(S,R)$ $C_n$-bimodule $N$ so that $(M,N)$ and $(N,M)$ are dual pairs, and the evaluation and coevaluation maps of these dual pairs give equivalences of $C_n$-spectra
\[
M\wedge^{\mathbb{L}}_{S}N \simeq R \hspace{.5cm} \textup{and} \hspace{.5cm} N\wedge^{\mathbb{L}}_{R}M \simeq S.
\]

\begin{thm}\label{cor-twisted-thh-morita}
Twisted topological Hochschild homology, $\THH_{C_n}$, is Morita invariant. In other words, for Morita equivalent $C_n$-ring spectra $R$ and $S$, $\THH_{C_n}(R) \simeq \THH_{C_n}(S).$
\end{thm}
\vspace{.3cm}

As noted in Remark \ref{rmk:notashadow}, although twisted topological Hochschild homology satisfies cyclic invariance, it is not a shadow as the coherence axioms for shadows do not hold. However, we show below that upon taking fixed points, it \emph{is} a shadow functor. This observation will be important in our later study of topological restriction homology.
\begin{prop}\label{prop:fixedpointsshadow}
The $C_n$-fixed points of twisted topological Hochschild homology, $\THH_{C_n}(R;-)^{C_n},$ extend to a shadow on the bicategory $\sR_{\cat{Sp}^{C_n}}$ of $C_n$-ring spectra and their bimodules. 
\end{prop}
\begin{proof}
We have proven cyclic invariance of $\THH_{C_n}$ in Proposition \ref{cor:twisted_morita}, and therefore the $C_n$-fixed points also satisfy cyclic invariance. It remains to show that the coherence properties hold. Let $A, B, C$ be $C_n$-ring spectra, $M$ an $(A,B)$-bimodule, $N$ a $(B,C)$-bimodule, and $P$ a $(C,A)$-bimodule. Let $g$ be the chosen generator of $C_n$. We study the hexagon diagram

\[
\xymatrix{\THH(A; {}^g(M\odot N)\odot P) \ar[r]^{\theta}\ar[d]_{\Sh{}{a}}& \THH(C; {}^g P\odot (M\odot N))
 \ar[r]^{\Sh{}{a}} & \THH(C; {}^g (P\odot M)\odot N) \\
\THH(A; {}^g M\odot (N\odot P)) \ar[r]^{\theta} & \THH(B; {}^g ( N\odot P)\odot M) \ar[r]^{\Sh{}{a}} & \THH(B; {}^g N\odot (P\odot M)) \ar[u]_{\theta}
}
\]

As in the proof of Proposition \ref{cor:twisted_morita}, the cyclic isomorphism $\theta$ in the top row is given by rotating the smash factors in the cyclic bar construction, and acting by $g$ on all of the $A$ and $P$ smash factors. Similarly, the cyclic isomorphism $\theta$ in the bottom row is given by rotating the smash factors in the cyclic bar construction, and acting by $g$ on the $A$ smash factors and the $\Bar{}(N;C;P) = N \odot P$ smash factor. Finally, the vertical $\theta$ is given by acting by $g$ on all $B$ smash factors, and on the $\Bar{}(P;A;M) = P \odot M$ smash factor. Upon taking $C_n$-fixed points, the action of $g$ becomes trivial, and the diagram commutes. 

\medskip 

Let $M$ be an $(A,A)$-bimodule. A similar analysis shows that upon taking $C_n$-fixed points, the diagram

\[
\xymatrix{\THH(A; {}^g M\odot U_A) \ar[r]^{\theta} \ar[dr]_-{\Sh{}{r}} & \THH(A; {}^g U_A \odot M) \ar[r]^{\theta} \ar[d]^{\Sh{}{l}} & \THH(A; {}^g M \odot U_A) \ar[dl]^{\Sh{}{r}} \\
& \THH(A; {}^g M) &
}
\]
also commutes.
\end{proof}
\begin{rmk}
While twisted topological Hochschild homology is not itself a shadow functor, we expect that it does fit into a framework of \emph{equivariant shadows}. We will return to the development of such a framework in subsequent work. The Hochschild homology for Green functors, described in Section \ref{subsec:HH_case} below, should similarly fit into the framework of an equivariant shadow.
\end{rmk}

\subsection{Hochschild homology for Green functors}
\label{subsec:HH_case}
The theory of $C_n$-twisted topological Hochschild homology has an algebraic analogue, a theory of Hochschild homology for Green functors. The twisted Hochschild homology of Green functors was developed in \cite{BlGeHiLa} and further studied in \cite{aghkk}. Below, we show that this Hochschild homology for Green functors can also be interpreted as a twisted Hochschild homology theory as in Definition \ref{def:phi-twisted_HH}.

One ingredient in the definition of twisted Hochschild homology for Green functors is the notion of norms for a Mackey functor. For $H\subset G$, where $G$ is finite, the norm functor $N_H^G$ maps an $H$-Mackey functor to a $G$-Mackey functor. It is defined to be a left Kan extension along the coinduction between the Burnside categories (\cite[2.3.2]{Hoy}) and is compatible with the norm functor for equivariant spectra.

Let $G$ be a finite cyclic group.
As discussed in Example \ref{exmp:Vaut},
for any $G$-Green functor $\m R$, a generator $g\in G$ defines an automorphism of $\m R$ by first passing to the Weyl group and then acting via the Weyl group action. 
\begin{defn}\label{twistaction}
Let $G$ be a finite cyclic group, and let $g \in G$. Let $\m{R}$ be a Green functor for $G$, and let $\m{M}$ be a left
$\m{R}$-module with action map $\lambda$. Let $\gmM$ denote $\m{M}$ with the module structure twisted by $g$. In other words the action map ${}^{g}\!\lambda$ is specified by the commuting diagram
$$
\xymatrix{\m{R} \Box \m{M} \ar[d]_{g \Box 1} \ar[dr]^{{}^{g}\!\lambda} & \\
\m{R} \Box \m{M} \ar[r]^{\lambda} & \m{M}.}
$$
\end{defn}

We can now define the twisted cyclic nerve. 
\begin{defn}
For an $\m{R}$-bimodule $\m{M}$, the \emph{twisted cyclic nerve} of $\m R$ with coefficients $\gmM$, denoted by $\HC^G_\bullet (\m R;{}^g\!\m M)$, is a simplicial Mackey functor with $q$ simplices given by
\[
\HC^G_q (\m R;{}^g\!\m M) = \gmM \Box\m{R}^{\Box q}.
\]
The face maps $d_i$ are given as usual by multiplication of the $i$th and $(i + 1)$st factors if $0 < i < q$.
The face map $d_0$ is the ordinary right module action map for $\m{M}$, while the last face map, $d_q$, rotates the last factor to the front and then uses the twisted left action map of Definition \ref{twistaction}.
The degeneracy maps $s_i$ are induced by inclusion of the unit after the $i$th factor, for $0 \leq i \leq q$. When the inputs $\m R$ and  $\m M$ are $H$-Green functors for a subgroup $H\subset G$, we define the \emph{$G$-twisted
cyclic nerve relative to $H$} to be the simplicial Mackey functor 
\[\m \HC^G_H (\m R;{}^g\!\m M)_\bullet:=\m \HC^G_\bullet(N_H^G\m R; {}^g\!N_H^G\m M).
\]
\end{defn}

We now recall the definition of twisted Hochschild homology for Green functors from \cite{BlGeHiLa}, extending the original definition to include coefficients. Recall that the homology of a simplicial Mackey functor is the homology of the associated normalized dg Mackey functor, given by postcomposing with the Dold-Kan equivalence from simplicial
abelian groups to chain complexes.

\begin{defn}
Let $H\subset G$ be finite cyclic groups. Let $\m{R}$ be a Green functor for $H$, and let $\m M$ be an $\m R$-bimodule. The \emph{$G$-twisted Hochschild homology of $\m{R}$ with coefficients in $\m M$} is defined to be the homology of the twisted cyclic nerve
$$
\m{\HH}_{H}^G(\m{R};\m M)_i = H_i(\m \HC_H^G(\m{R};{}^g\!\m{M})_\bullet).
$$
When  $\m{M} = \m R$ with the canonical bimodule structure, we write $\m{\HH}_{H}^G(\m{R})$ for $\m{\HH}_{H}^G(\m{R}, \m{R})$,  the \emph{$G$-twisted Hochschild homology of $\m{R}$}. When $H=G$, we abbreviate the notation and write $\m{\HH}^G_i(\m{R};\m M)$ and $\m{\HH}^G_i(\m{R}).$
\end{defn}

To interpret $G$-twisted Hochschild homology as the $g$-twisted Hochschild homology defined in Definition \ref{def:phi-twisted_HH}, we work in the category of simplicial Mackey functors. 
The category of simplicial Mackey functors, denoted by $s\Mack$, is a symmetric monoidal simplicial model category by \cite[Theorem 4.3]{BlGeHiLa}. Let $\m R$ be a $G$-Green functor and let $\m M$ be a $\m R$-bimodule. We view $\m R$ as a constant simplicial Green functor denoted by $\mathrm{const}(\m{R})$, and similar for $\m M$. The automorphism $g$ induces an automorphism of $\mathrm{const}(\m{R})$.
From Definition \ref{def:phi-twisted_HH},
$$\HH^g(\mathrm{const}(\m{R}); \mathrm{const}(\m{M}))\cong \lvert B^\text{cyc}_{\bullet}( \mathrm{const}(\m{R});{^g\!\mathrm{const}(\m{M})})\lvert.$$

The cyclic bar construction defines a bisimplicial Mackey functor.
Since $\Mack$ is a cocomplete category, for any bisimplicial Mackey functor $M_{\bullet,\bullet}$, the associated diagonal simplicial Mackey functor $d M_{\bullet,\bullet}$ is equivalent to the geometric realization $\lvert M_{\bullet,\bullet}\lvert $.\footnote{For simplicial sets, this equivalence is a standard result (see e.g. \cite[p. 210]{SHT}). For bisimplicial Mackey functors, a similar proof works by checking the diagonal has the same universal property as this geometric realization.} 
The diagonal simplicial Mackey functor associated to $B^\text{cyc}_{\bullet}( \mathrm{const}(\m{R});{^g\!\mathrm{const}(\m{M})})$ is $\m \HC_\bullet (\m R; {}^g\!\m M)$.
Hence we have 
$$\HH^g(\mathrm{const}(\m{R});\mathrm{const}(\m{M}))\cong \lvert B^\text{cyc}_{\bullet}( \mathrm{const}(\m{R});{^g\!\mathrm{const}(\m{M})})\lvert\cong \m \HC^G_\bullet (\m R; {}^g\!\m M).$$

Thus, the $G$-twisted Hochschild homology of Green functors is the homology of a $g$-twisted Hochschild homology.

We now consider the bicategory $\sR_{\cat {sMack}_{C_n}}$, which is an example of the bicategory $\sR_{\cat V}$ in Definition \ref{def:bicat-V} for $\cat V = \cat{sMack}_{C_n}$, the category of simplicial $C_n$-Mackey functors.

The framework developed here allows us to conclude the twisted Hochschild homology of Green functors is Morita invariant. Given two $H$-Green functors $\m R$ and $\m S$, we consider them as constant simplicial $H$-Green functors. Unwinding Definition \ref{def:Morita}, we see that $\m R$ and $\m S$ are Morita equivalent if there exist an $(\m R, \m S)$-bimodule $\m M \in s\text{Mack}_H$ and an $(\m S, \m R)$-bimodule $\m N \in s\text{Mack}_H$ so that $(\m M, \m N)$ and $(\m N,\m M)$ are dual pairs, and the evaluation and coevaluation maps of these dual pairs give equivalences of simplicial Mackey functors 

$$|\Bar (\m M;\m S; \m N)| \simeq \m R \hspace{.5cm} \textup{ and } \hspace{.5cm}|\Bar (\m N; \m R; \m M)| \simeq \m S$$

\begin{prop}\label{prop:MackeyMorita}
For a finite cyclic group $G$, the $G$-twisted Hochschild homology of $G$-Green functors, $\m \HH_i^G$, is Morita invariant. That is, if $G$-Green functors $\m R$ and $\m S$ are Morita equivalent in the sense defined above, then $\m \HH_i^G (\m R) \cong \m \HH_i^G (\m S)$ for all $i$.
\end{prop}
\begin{proof}
The result follows from Proposition \ref{cor:twisted_morita} and the identification that $\m \HC^G_\bullet (\m R; {}^g\!\m R)=\HH^g(\m{R})$ after passing to homology.
\end{proof}

We can further extend this result to the $G$-twisted Hochschild homology of an $H$-Green functor, where $H \subset G$ are finite cyclic groups.

\begin{thm}\label{cor-HH-Green-morita}
For a finite cyclic group $G$, and $H \subset G$, the $G$-twisted Hochschild homology for $H$-Green functors is Morita invariant. That is, if $H$-Green functors $\m R$ and $\m S$ are Morita equivalent in the sense defined above, then $\m \HH_{H}^G (\m R)_i \cong \m \HH_{H}^G (\m S)_i$ for all $i$.
\end{thm}

\begin{proof}

Denote by $N^G_H: s\text{Mack}_H\to s\Mack$ the induced functor obtained by applying $N^G_H$ levelwise.

We first prove cyclic invariance. Let $\m{A}$ and $\m{B}$ be two simplicial $H$-Green functors.
Let $\m{M}$ be an $(\m{A},\m{B})$-bimodule and $\m{N}$ a $(\m{B},\m{A})$-bimodule. Note that $N_H^G\m{M}$ is an $(N_H^G\m{A},N_H^G\m{B})$-bimodule and similar for $N_H^G \m{N}.$

We have that 
\begin{equation}
    \label{eq:norm}
\begin{aligned}
	N^G_H (\m{M}\odot \m{N})
&\stackrel{\text{def}}{=} 
N_H^G \lvert \Bar (\m{M};\m{B}; \m{N})\lvert\\
    &\stackrel{\text{(1)}}{\cong} \lvert N_H^G \Bar (\m{M};\m{B}; \m{N})\lvert\\
    &\stackrel{\text{(2)}}{=} \lvert  \Bar (N_H^G \m{M};N_H^G \m{B}; N_H^G \m{N})\lvert
    \stackrel{\text{def}}{=} N_H^G\m{M} \odot N_H^G \m{N}.
\end{aligned}
\end{equation}
Isomorphism (1) holds since $N_H^G$ preserves sifted colimits, while isomorphism (2) holds since $N_H^G$ is strong symmetric monoidal.

By Propositions \ref{cor:twisted_morita} and \ref{prop:MackeyMorita}, $\HH^g$ satisfies cyclic invariance, and thus we have 
$$\HH^g(N_H^G\m{A}; N_H^G\m{M} \odot N_H^G \m{N})\simeq \HH^g(N_H^G\m{B}; N_H^G\m{N} \odot N_H^G \m{M}).$$
Combining with equation (\ref{eq:norm}), we have 
$$\HH^g(N_H^G\m{A}; N_H^G(\m{M} \odot \m{N}))\simeq \HH^g(N_H^G\m{B}; N_H^G(\m{N} \odot \m{M})).$$
Thus $\m \HH_{H}^G$ satisfies cyclic invariance. As in the proof of Proposition \ref{cor:twisted_morita} it then follows that $\m \HH_{H}^G$ is Morita invariant.

\end{proof}

\begin{rmk}
The proof of Theorem \ref{cor-HH-Green-morita} requires that the norm functor is strong symmetric monoidal and commutes with geometric realizations. In general, any such functor $F$ preserves cyclic invariance, and hence preserves Morita invariance as in Theorem \ref{cor-HH-Green-morita} above.
\end{rmk}

\begin{rmk} \label{rmk-many-obj-equivar}
Proceeding as in Section \ref{subsec:manyobj}, one can also define many-object versions of $\THH_{C_n}$ and $\m{\HH} ^{C_n}$. These will be $\HH_{\cat V _{\mathsf{aut}}\mathsf{Cat}}$ for $\cat V = \cat{Sp}^{C_n}$ and $\cat V = s\mathsf{Mack}_{C_n}$ respectively. To prove that these are Morita invariant via this approach, one would need to verify that Assumption \ref{assumption:manyobj} holds for these categories. We expect that $s\mathsf{Mack}_{C_n}$ is a locally presentable base, and therefore satisfies Assumption \ref{assumption:manyobj}; however, verifying Assumption \ref{assumption:manyobj} is likely to be more difficult for $\cat{Sp}^{C_n}$.
\end{rmk}

We finish this section by noting that although $\m{\HH}_H^{G}$ does not have the full structure of a shadow functor, when evaluated at the orbit $(G/G)$ it is indeed a shadow. This observation will play an important role in our study of algebraic restriction homology. 

\begin{prop}\label{prop-HH-Green-fp-shadow}  Let $G$ be a finite cyclic group, and $H \subset G$. On the bicategory of simplicial $H$-Green functors and bimodules, the $G$-twisted Hochschild homology evaluated at the orbit $G/G$, $\m \HH_{H}^{G}(\m{A};-)_i(G/G)$, extends to a shadow.
 \end{prop}
 \begin{proof}
 We have shown cyclic invariance in Theorem \ref{cor-HH-Green-morita}; it remains to show that the coherence diagrams commute. Let $\m{A},\m{B},\m{C}$ be simplicial $H$-Green functors, $\m{M}$ an $(\m{A},\m{B})$-bimodule, $\m{N}$ a $(\m{B},\m{C})$-bimodule, and $\m{P}$ a $(\m{C},\m{A})$-bimodule. We study the hexagon diagram \\
 \adjustbox{scale=0.83}{$$
\xymatrix{|B^\text{cyc}_{\bullet}(N_H^G \m{A}; {}^g N_H ^G ((\m{M}\odot \m{N})\odot \m{P}))| \ar[r]^{\theta}\ar[d]_{\Sh{}{a}}& |B^\text{cyc}_{\bullet}(N_H ^G \m{C}; {}^g N_H ^G (\m{P}\odot (\m{M}\odot \m{N})))|
 \ar[r]^{\Sh{}{a}} & |B^\text{cyc}_{\bullet}(N_H ^G \m{C}; {}^g N_H ^G ((\m{P}\odot \m{M})\odot \m{N}))| \\
 |B^\text{cyc}_{\bullet}(N_H ^G \m{A}; {}^g N_H ^G (\m{M}\odot (\m{N}\odot \m{P})))| \ar[r]^{\theta} & |B^\text{cyc}_{\bullet}(N_H ^G \m{B}; {}^g N_H ^G (( \m{N}\odot \m{P})\odot \m{M}))| \ar[r]^{\Sh{}{a}} & |B^\text{cyc}_{\bullet}(N_H ^G \m{B}; {}^g N_H ^G (\m{N}\odot (\m{P} \odot \m{M}))) |. \ar[u]_{\theta}
}$$

}

As in the proofs of Proposition \ref{cor:twisted_morita} and Theorem \ref{cor-HH-Green-morita}, the cyclic isomorphism $\theta$ in the top row is given by rotating the smash factors in the cyclic bar construction, and acting by $g$ on all of the $N_H ^G \m{A}$ and $N_H ^G \m{P}$ smash factors. Similarly, the cyclic isomorphism $\theta$ in the bottom row is given by rotating the smash factors in the cyclic bar construction, and acting by $g$ on the $N_H ^G\m{A}$ smash factors and the $N_H ^G \Bar{}(\m{N};\m{C};\m{P}) = N_H ^G (\m{N} \odot \m{P})$ smash factor. Finally, the vertical $\theta$ is given by acting by $g$ on all the $N_H ^G \m{B}$ smash factors, and on the $N_H ^G \Bar{}(\m{P};\m{A};\m{M}) = N_H ^G (\m{P} \odot \m{M})$ smash factor. Recall that the action of $g$ on a $G$-Mackey functor is induced by the Weyl group action; thus, once we evaluate at $G/G$, the Weyl group action becomes trivial and so does the action of $g$. Therefore, after evaluating at $G/G$, the diagram commutes. 

\medskip 

Let $\m{M}$ be an $(\m{A},\m{A})$-bimodule. A similar analysis shows that upon evaluating at $G/G$, the diagram\\
\adjustbox{scale=0.93}{
$$
\xymatrix{|B^\text{cyc}_{\bullet}(N_H ^G \m{A}; {}^g N_H ^G (\m{M}\odot U_{\m{A}}))| \ar[r]^{\theta} \ar[dr]_-{\Sh{}{r}} & |B^\text{cyc}_{\bullet}(N_H ^G \m{A}; {}^g N_H ^G (U_{\m{A}} \odot \m{M}))| \ar[r]^{\theta} \ar[d]^{\Sh{}{l}} & |B^\text{cyc}_{\bullet}(N_H ^G \m{A}; {}^g N_H ^G (\m{M} \odot U_{\m{A}}))| \ar[dl]^{\Sh{}{r}} \\
& |B^\text{cyc}_{\bullet}(N_H ^G \m{A}; {}^g N_H ^G \m{M})| &
}$$
}\\
also commutes.
 
 \end{proof}

\subsection{Algebraic $tr$.} In the classical approach to trace methods in algebraic $K$-theory, one analyzes the fixed points of topological Hochschild homology, or TR-theory. For a ring $A$, TR$^n(A;p)$ is defined to be $\THH(A)^{C_{p^{n-1}}}.$ The spectrum $\TR(A;p)$ is then defined to be 
\[
\TR(A;p) := \holim_{\substack{\longleftarrow\\R}} \textup{TR}^n(A;p),
\]
where $R$ is the restriction map on $\THH(A)^{C_{p^{n-1}}}$, which can be defined using the cyclotomic structure on THH. In \cite{BlGeHiLa}, the authors introduce an algebraic analog of TR-theory, using twisted Hochschild homology for Green functors. For a ring $A$, they define this algebraic analogue as 
\[
tr_k(A;p) := \lim_{\longleftarrow} \m{\HH}_{e}^{C_{p^{n}}}(A)_k(C_{p^{n}}/C_{p^{n}}),
\]
where the maps in the limit are algebraic analogues of the restriction map, defined using a cyclotomic structure on twisted Hochschild homology for Green functors. 

In this paper, we extend the definition of algebraic $tr$ of rings, $tr(A;p)$, to $tr$ of (simplicial) rings with coefficients in a bimodule, $tr(A,M;p)$. We will show that this is a bicategorical shadow. We fix a prime $p$ throughout and omit the prime from the notation.

To define $tr(A,M)$, we follow the definition of $tr(A)$ in \cite{BlGeHiLa}. Let $A$ be a simplicial ring and $M$ an $A$-bimodule. As in Proposition 5.17 of \cite{BlGeHiLa}, there is a natural isomorphism
$$\Phi^{C_p}(\m{\HC}^{C_{p^n}} _e (A;M)_k) \cong \Phi^{C_p}((N^{C_{p^n}} _e M)\Box(N^{C_{p^n}} _e A)^{\Box k} ) \cong (N^{C_{p^{n-1}}} _e M) \Box(N^{C_{p^{n-1}}} _e A)^{\Box k}_.  $$
Here $\Phi^{C_p}$ denotes the geometric fixed points for Mackey functors, as defined in \cite{BlGeHiLa}. The geometric fixed points functor takes the twisting by a generator to the twisting by a generator. Thus $\Phi^{C_p}(\m{\HC}^{C_{p^n}} _e (A;M)_k)$ is isomorphic to $\m{\HC}^{C_{p^{n-1}}} _e (A;M)_k$. Note that this is an algebraic analogue of Proposition 7.5 of \cite{CLMPZ}.

Let $N \subset G$ be a normal subgroup, and let $\m{A}$ denote the Burnside Mackey functor for $G$. As in \cite{BlGeHiLa}, let $E\mathcal{F}_N(\m{A})$ denote the sub-Mackey functor of $\m{A}$ generated by $\m{A}(G/H)$ for all subgroups $H$ that do not contain $N$. Let $E\widetilde{\mathcal{F}}_N(\m{A})$ denote $\m{A}/E\mathcal{F}_N(\m{A})$. For a simplicial $G$-Mackey functor $\m{M}$, let
\[
E\widetilde{\mathcal{F}}_N(\m{M}):= \m{M} \Box E\widetilde{\mathcal{F}}_N(\m{A}). 
\]
Let $G$ be finite cyclic group that contains $C_p$, and let $\m{M}$ be a simplicial $G$-Mackey functor. As in the proof of Corollary 5.18 of \cite{BlGeHiLa}, there is a natural isomorphism of simplicial abelian groups 
$$\widetilde{E}\mathcal{F}_{C_p} \m{M} (G/G) \cong (\Phi^{C_p}\m{M})\big((G/C_p)/(G/C_p)\big).$$
The natural map of simplicial Mackey functors
$$ \m{\HC}^{C_{p^n}} _e (A;M) \to \widetilde{E}\mathcal{F}_{C_p} \m{\HC}^{C_{p^n}} _e (A;M)$$
therefore induces a natural map 
$$\m{\HC}^{C_{p^n}} _e (A;M)(C_{p^n}/C_{p^n}) \to (\Phi^{C_p}\m{\HC}^{C_{p^n}} _e (A;M))(C_{p^{n-1}}/C_{p^{n-1}}) \cong \m{\HC}^{C_{p^{n-1}}} _e (A;M)(C_{p^{n-1}}/C_{p^{n-1}}).$$
This is our algebraic restriction map, which induces in turn a map
$$r: \m{\HH}^{C_{p^n}} _e (A;M)_k(C_{p^n}/C_{p^n}) \to \m{\HH}^{C_{p^{n-1}}} _e (A;M)_k(C_{p^{n-1}}/C_{p^{n-1}})$$
for each $k$ after passing to homology.

\begin{defn}\label{defn-tr} Algebraic $tr$ with coefficients is obtained by taking the inverse limit along the restriction maps $r$: 
$$tr_k(A,M) := \lim_{\substack{\longleftarrow\\r}} \m{\HH}_{e}^{C_{p^{n}}}(A;M)_k(C_{p^{n}}/C_{p^{n}}).$$
\end{defn}

As noted in Remark \ref{rmk:notashadow}, although Hochschild homology for Green functors satisfies cyclic invariance, it is not a shadow as the coherence axioms for shadows do not hold. Interestingly, algebraic restriction homology, $tr$, \emph{is} a shadow functor.

\begin{prop}\label{prop-tr-shadow}
The functor $tr_k$ extends to a shadow on the bicategory of  simplicial rings and their bimodules, landing in the category of abelian groups.
\end{prop}

\begin{proof} 
Each $\m{\HH}_{e}^{C_{p^{n}}}(A;-)_k(C_{p^n}/C_{p^n})$ is a shadow (Proposition \ref{prop-HH-Green-fp-shadow}). The restriction maps $r$ are natural maps which respect the cyclic isomorphisms, and therefore the inverse limit is also a shadow.
\end{proof}

As shadows respect Morita equivalences, we can conclude that algebraic $tr$ is Morita invariant.

\begin{cor}\label{cor-tr-morita} Let $(\mab MAB, \mab NBA)$ be a Morita equivalence of simplicial rings $A$ and $B$.  For any $A$-bimodule $Q$, there is an isomorphism
\[
tr(A,Q) \xrightarrow{\cong} tr(B,N \otimes^\mathbb{L}_A Q\otimes^\mathbb{L}_A M).
\]
where $\otimes^\mathbb{L}$ denotes derived tensor product, computed using the bar construction. In particular, the Euler characteristic
\[
\chi(M)\colon tr(A) \to tr(B)
\]
is an isomorphism.
\end{cor}

%% file: laxshadow.tex
\section{Lax shadow functors.}\label{sec:laxshadow}

In this section, we show that the linearization maps from $\THH$ and $\TR$ to their algebraic counterparts are morphisms of shadows. That is, we show that these linearization maps arise from lax shadow functors in the sense of \cite{PS}.

\subsection{Linearization as a lax shadow functor}
Recall that classically, for $A$ a ring, there is a linearization map
$$\pi_k \THH(HA) \to \HH_k(A)$$
relating topological Hochschild homology and ordinary Hochschild homology. In this section we show that this linearization map arises from a morphisms of shadows between the topological and algebraic theories. That is, it arises as a lax shadow functor. We further construct a new linearization map, from topological restriction homology (TR) to algebraic restriction homology ($tr$), \[
\pi_k\TR(HA) \to tr_k(A),
\]
and prove that it also arises as a lax shadow functor.

We first recall the definitions of a lax functor and a lax shadow functor from \cite{PS}.

\begin{defn}
Let $\sB$ and $\sC$ be bicategories. A lax functor $F: \sB \to \sC$ consists of:
\begin{itemize}
    \item a function $F_0$ from the objects of $\sB$ to the objects of $\sC$
    \item functors $F_{R,S}: \sB(R,S) \to \sC\big(F_0(R), F_0(S)\big)$ for every pair of objects $R,S$ in $\sB$, and
    \item natural transformations
$$c: F_{R,S}(M) \odot F_{S,T}(N) \to F_{R,T}(M \odot N)$$
$$i: U_{F(R)} \to F(U_R)$$
\end{itemize}
satisfying appropriate coherence axioms. For details on the coherence axioms, see Definition 4.1 of \cite{Ben}. For illuminating pictures, see section 8 of \cite{PS}.
\end{defn}

\begin{defn}\label{defn:laxshadow} Let $\sB$ and $\sC$ be bicategories, and let
\[
\Sh{\sB}{-}\colon \coprod_{ A\in \sB_0} \sB(A,A) \to \cat T  
\]

\[
\Sh{\sC}{-}\colon \coprod_{D\in \sC_0} \sC(D,D) \to \cat Z  
\]
be shadows. A lax shadow functor consists of a lax functor $F: \sB \to \sC$ along with a functor $F_{tr}: \cat T \to \cat Z$, and a natural transformation
$$\phi: \Sh{\sC}{-} \circ F \to F_{tr} \circ \Sh{\sB}{-}$$
such that the following diagram commutes whenever it makes sense.

$$\xymatrix{
\Sh{}{F(M) \odot F(N)} \ar[r]^-\theta \ar[d] & \Sh{}{F(N) \odot F(M)} \ar[d] \\
\Sh{}{F(M \odot N)}  \ar[d]^\phi & \Sh{}{F(N \odot M)} \ar[d]^\phi \\
F_{tr} \Sh{}{M \odot N} \ar[r]^-{F_{tr}(\theta)} & F_{tr} \Sh{}{N \odot M}
}$$
\end{defn}

\begin{rem}\label{rem-lax-duals}
Note that a lax shadow functor is not guaranteed to preserve dual pairs. In Proposition 8.3 of \cite{PS}, it is shown that a lax shadow functor $F$ will preserve dual pairs if the map $F(M) \odot F(N) \to F( M \odot N)$ is an isomorphism. This is rarely true for the example we consider below, the Eilenberg--MacLane functor $H: \cat{Ab} \to \cat{Sp}$ (or $H: \cat{Mack}_{C_n} \to \cat{Sp}^{C_n}$).
\end{rem}

We first consider a generalization of the Eilenberg--MacLane spectrum functor to simplicial abelian groups. If $A_\bullet$ is a simplicial abelian group, define $H(A_\bullet) = | H A_\bullet|$, i.e., the geometric realization of the simplicial spectrum whose $i^{th}$ level is $HA_i$.

\begin{prop}\label{prop-EM-lax} The Eilenberg--MacLane spectrum functor $H: s\cat{Ab} \to \cat{Sp}$ defines a lax shadow functor between the shadows $\pi_k \THH$ and $\HH_k$. In particular, there is a natural transformation 
$$\phi_{A,M}: \pi_k \THH(HA; HM) \to \HH_k(A;M)$$
and a functor $F_{tr}= id: \cat{Ab} \to \cat{Ab}$, satisfying appropriate coherence conditions.
\end{prop}

\begin{proof}
Let $H$ be the Eilenberg--MacLane functor from the bicategory of simplicial rings, bimodules, and bimodule maps to the bicategory $\sR_{\cat Sp}$ of ring spectra,  bimodule spectra and their maps. (The latter is the bicategory $\sR_{\cat V}$ from Definition \ref{def:bicat-V}, for $\cat V$ the category of spectra. The former is $\sR_{s\cat{Ab}}$.) Since $H$ is a lax symmetric monoidal functor from (simplicial) abelian groups to spectra, it induces a lax functor on the aforementioned bicategories. That is:
\begin{itemize}
    \item if $A$ is a (simplicial) ring, then $HA$ is a ring spectrum;
    \item if $M$ is a $(B,A)$-bimodule, then $HM$ is an $(HB,HA)$-bimodule;
    \item for $(B,A)$-bimodules $M$ and $(A,B)$-bimodules $N$, there are natural transformations $c: |\Bar{(HM; HA; HN)}|  \to H|\Bar{(M;A;N)}|$ that satisfy appropriate coherence conditions, and
    \item the unit natural transformation $i: HA \to HA$ is the identity map.
\end{itemize}
We take $F_{tr}: \cat{Ab} \to \cat{Ab}$ to be the identity. Thus we need to define a natural transformation
$$\phi_{A,M}: \pi_k \THH(HA; HM) \to \HH_k(A;M)$$
that satisfies the appropriate compatibility condition with the shadow structure.

We define $\phi_{A,M}$ as follows. The $p^{th}$ level in the cyclic bar construction for $\THH(HA; HM)$ is given by $HM \sm (HA)^{ \sm p}$. Since $H$ is lax symmetric monoidal, there is a natural map
$$HM \sm (HA)^{\sm p}  \to H(M \otimes A^{\otimes p})$$
which respects the face and degeneracy maps. Thus we obtain a map on the geometric realizations
$$|HM \sm (HA)^{ \sm \bullet}| \to |H(M \otimes A^{\otimes \bullet})|.$$
The $k^{th}$ homotopy group of the left hand side is $\pi_k \THH(HA; HM)$. The $0^{th}$ space of the right hand side is the topological abelian group $B^\text{cyc}(A;M)$, and its $i^{th}$ space is $B^i B^\text{cyc}(A;M)$; here, $B^i$ denotes the $i$-fold bar construction. The $k^{th}$ homotopy group of the right hand side is $\HH_k(A;M)$. The map of geometric realizations thus induces a natural transformation $\phi_{A,M}: \pi_k \THH(HA; HM) \to \HH_k(A;M)$.

The Eilenberg--MacLane functor $H$ is lax symmetric monoidal, and the cyclic isomorphisms $\theta$ for $\THH$ and $\HH$ are both obtained on the cyclic bar construction. Therefore the diagram
$$\xymatrix{
\pi_k \THH(HA; |\Bar (HM ;HB ; HN)|) \ar[r]^-\theta \ar[d] & \pi_k \THH(HB; |\Bar (HN ; HA ; HM)|) \ar[d] \\
\pi_k \THH (HA; H |\Bar(M ; B; N)|) \ar[d]^\phi & \pi_k \THH (HB; H(|\Bar (N ; A; M)|) \ar[d]^\phi \\
\HH_k(A; |\Bar(M; B; N)|) \ar[r]^-{\theta} & \HH_k(B; |\Bar (N ;A ; M)|)
}$$
commutes for all $(A,B)$-bimodules $M$ and $(B,A)$-bimodules $N$.
\end{proof}

For any $(-1)$-connected ring spectrum $R$, there is a map of ring spectra $R \to H\pi_0(R)$ that is an isomorphism on $\pi_0$ and induces 0 on $\pi_i$ for $i > 0$. 
The linearization map
$$\pi_k \THH(R) \to \HH_k(\pi_0 R)$$
is simply the composite
$$\xymatrix{
\pi_k \THH(R) \ar[r] & \pi_k \THH(H\pi_0 R) \ar[r]^-\phi & \HH_k(\pi_0 R).
}$$
This construction generalizes to the $C_n$-equivariant case. We first generalize the equivariant Eilenberg--MacLane spectrum functor to simplicial Mackey functors. Let $H:\mathsf{Mack}_{C_n} \to \mathsf{Sp}^{C_n}$ denote the equivariant Eilenberg--MacLane spectrum functor; if $\m{M}_\bullet$ is a simplicial Mackey functor, define $H(\m{M}_\bullet) = |H\m{M}_\bullet|$, the geometric realization of the simplicial $C_n$-spectrum whose $i^{th}$ level is $H\m{M}_i$.

\begin{prop}\label{prop-EM-realiz}
If $\m{R}$ is a simplicial $C_n$-Green functor, and $\m{M}$ is an $\m{R}$-bimodule, then
$$\pi_k (|H({}^g \m{M}\Box\m{R}^{\Box \bullet} )|^{C_n}) \cong \m{\HH}^{C_n}_k (\m{R}; \m{M})(C_n/C_n).$$
\end{prop}

\begin{proof}
Denote by $\mathcal{G}$ the orbit category of the group $G=C_n$, with objects $G/K$ and morphisms $G$-maps. A $\mathcal{G}$-space is a functor from $\mathcal{G}^{op}$ to topological spaces. For example, a Mackey functor $\m{M}$ defines a $\mathcal{G}$-space whose value on $G/K$ is the discrete abelian group $\m{M}(G/K)$. A simplicial Mackey functor $\m{M}_\bullet$ defines a $\mathcal{G}$-space whose value on $G/K$ is the geometric realization $|\m{M}_\bullet (G/K)|$.

In \cite{dSN}, the authors study a functor $\Psi$ from $\mathcal{G}$-spaces to $G$-spaces that is a variant of the Elmendorf coalescence functor. In Proposition 4.2 of \cite{dSN}, they show that $\Psi$ is an inverse up to weak equivalence of the functor from $G$-spaces to $\mathcal{G}$-spaces that takes a $G$-space $X$ to the $\mathcal{G}$-space $G/K \mapsto X^K$. It follows that for every $\mathcal{G}$-space $\mathcal{X}$, there is a natural weak equivalence $\Psi(\mathcal{X})^K \to \mathcal{X}(G/K)$. In particular, when $\mathcal{X} = {}^g \m{M} \Box\m{R}^{\Box p}$,
there is a natural weak equivalence
$$\xymatrix{
\Psi({}^g \m{M} \Box\m{R}^{\Box p} )^K \ar[r]^-{\sim} & ({}^g \m{M} \Box\m{R}^{\Box p}  )(G/K).
}$$

By Theorem 4.22 of \cite{dSN}, the 0-space of the spectrum $H({}^g \m{M} \Box\m{R}^{\Box p})$ is naturally equivalent to $\Psi({}^g \m{M} \Box\m{R}^{\Box p} )$. Let $H(0)$ denote the 0-space of the spectrum $|H({}^g \m{M} \Box\m{R}^{\Box \bullet})|$. We apply geometric realization to the equivalence above, using the fact that our groups are all finite, and geometric realization commutes with finite limits.  We obtain
$$\xymatrix{
H(0)^K  \ar[r]^-{\sim} & |\m{HC}^{C_n} _\bullet(\m{R}; \m{M})(G/K)|
}$$
and $\pi_k (|\m{HC}^{C_n} _\bullet(\m{R}; \m{M})(G/K)|) \cong \underline{H}_k(\m{HC}^{C_n} (\m{R}; \m{M}))(G/K)  = \m{\HH}^{C_n}_k (\m{R}; \m{M})(G/K)$, as required.
\end{proof}

\begin{prop}\label{prop-equiv-EM-lax} The Eilenberg--MacLane spectrum functor $H: s\cat{Mack}_{C_n} \to \cat{Sp}^{C_n}$ induces a lax shadow functor between the shadows $\pi_k \THH_{C_n}(-;-)^{C_n}$ and $\m{\HH}^{C_n}(-;-)_k(C_n/C_n)$. In particular, there is a natural transformation
\[\phi_{\m{R}, \m{M}} : \pi_k \THH_{C_n} (H\m{R}; H\m{M})^{C_n} \to \m{\HH}^{C_n}_k (\m{R}; \m{M})(C_n/C_n)\]
and a functor $F_{tr} = id: \cat {Ab} \to \cat {Ab}$, satisfying appropriate coherence conditions.
\end{prop}

\begin{proof}
In the notation of Definition \ref{defn:laxshadow}, let $F = H$ and $F_{tr} = \Id: \cat {Ab} \to \cat {Ab}$. Let $g$ be the generator $e^{2\pi i/n}$ of $C_n$. The functor $H$ is lax symmetric monoidal, so it induces a lax functor between the bicategory of simplicial $C_n$-Green functors and the bicategory of $C_n$-ring spectra. That is:
\begin{itemize}
   \item if $\m R$ is a (simplicial) $C_n$-Green functor, then $H \m R$ is a $C_n$-ring spectrum;
    \item if $\m M$ is an $(\m S,\m R)$-bimodule, then $H \m M$ is an $(H \m S,H \m R)$- bimodule;
   \item for $(\m S,\m R)$-bimodules $\m M$ and $(\m R,\m S)$-bimodules $\m N$, there are natural transformations $c: |\Bar (H \m M ;H \m R ; H \m N)| \to H | \Bar (\m M ; \m R;  \m N)|$ that satisfy appropriate coherence conditions, and
    \item the unit natural transformation $i: H\m R \to H \m R$ is the identity map. 
   \end{itemize}
For any $\m{R}$-bimodule $\m{M}$, $H({}^g \m{M}) = {}^g H\m{M}$. On each level of the cyclic bar construction, there is therefore a natural map
$${}^g H\m{M}\sm(H\m{R})^{\sm p}  \to  H({}^g \m{M} \Box \m{R}^{\Box p}  )$$
that commutes with the face and degeneracy maps and thus induces a map of $C_n$-spectra $$\THH_{C_n} (H\m{R}; H\m{M}) = |{}^g H\m{M} \sm(H\m{R})^{\sm \bullet}  | \to  |H({}^g \m{M}\Box\m{R}^{\Box \bullet}  )|.$$
By Proposition \ref{prop-EM-realiz}, applying $C_n$-fixed points and $\pi_k$ results in the natural transformation \[\phi_{\m{R}, \m{M}} : \pi_k \THH_{C_n} (H\m{R}; H\m{M})^{C_n} \to \m{\HH}^{C_n}_k (\m{R}; \m{M})(C_n/C_n)\]
As in the proof of Proposition \ref{prop-EM-lax}, the relevant diagram commutes, whence this a lax shadow functor, as desired.
\end{proof}

\subsection{\texorpdfstring{$\TR$}{TR} and algebraic \texorpdfstring{$tr$}{tr}.}
Recall from Proposition \ref{prop-tr-shadow} that algebraic $tr$ is a shadow on the bicategory of simplicial rings and their bimodules. We will show that the Eilenberg--MacLane functor $H: \cat{Ab} \to \cat{Sp}$ defines a lax shadow functor between $\TR$ and algebraic $tr$. For this, we will need to discuss $\TR$ with coefficients, first defined by Lindenstrauss--McCarthy \cite{LM}. We recall the definition of $\TR$ with coefficients from \cite{CLMPZ}. For a ring spectrum $R$ and an $R$-bimodule $M$, $p^n$-fold $\THH$ is defined as a $C_{p^n}$-spectrum:
$$\THH^{(p^n)}(R; M) : = \THH_{C_{p^n}} (N^{C_{p^n}} _e R; N^{C_{p^n}} _e M).$$

By Proposition 7.5 of \cite{CLMPZ}, when $R$ and $M$ are cofibrant,
$$\Phi^{C_p} \THH^{(p^n)}(R; M) \cong \THH^{(p^{n-1})}(R; M).$$
This is then used to define restriction maps
$$R: \THH^{(p^n)}(R; M)^{C_{p^n}} \to \THH^{(p^{n-1})}(R; M)^{C_{p^{n-1}}}.$$

\begin{defn}
The spectrum $\TR(R;M)$ is defined to be the homotopy inverse limit
\[
\TR(R;M) := \holim_{\substack{\longleftarrow\\R}} \THH^{(p^n)}(R; M)^{C_{p^n}}
\]
\end{defn}

The proof of the following proposition was adapted from unpublished work of Campbell, Lind, Malkiewich, Ponto, and Zakharevich:

\begin{prop}\label{prop-TR-shadow}
$\TR$ with coefficients, $\TR(R;-)$, extends to a shadow on the bicategory $\sR_{\cat Sp}$ of ring spectra and their bimodules. 
\end{prop}

\begin{proof}
Let $A$ and $B$ be ring spectra, $M$ an $(A,B)$-bimodule and $L$ a $(B,A)$-bimodule. We will construct cyclic isomorphisms
$$\tau^n _{A,B}: \THH^{(p^n)}(A; |\Bar(M;B;L)|)^{C_{p^n}} \cong \THH^{(p^n)}(B; |\Bar(L;A;M)|)^{C_{p^n}}$$
that are compatible with the restriction maps. Therefore we will obtain the desired equivalences
$$\theta_{A,B}: \TR(A; |\Bar(M;B;L)|) \to \TR(B; |\Bar(L;A;M)|).$$
First note that
 $$\THH^{(p^n)}(A; |\Bar(M;B;L)|) \cong  \THH_{C_{p^n}} (N_e ^{C_{p^n}} A; |\Bar(N_e ^{C_{p^n}} M; N_e ^{C_{p^n}} B; N_e ^{C_{p^n}} L)|). $$
 By Theorem \ref{cor-twisted-thh-morita}, we have $C_{p^n}$-equivariant isomorphisms
 $$\sigma^n _{A,B}: \THH^{(p^n)}(A; |\Bar(M;B;L)|) \cong \THH^{(p^n)}(B; |\Bar(L;A;M)|).$$
 Taking $C_{p^n}$-fixed points, we obtain the isomorphisms
 $$\tau^n _{A,B}: \THH^{(p^n)}(A; |\Bar(M;B;L)|)^{C_{p^n}} \cong \THH^{(p^n)}(B; |\Bar(L;A;M)|)^{C_{p^n}}.$$
 As in Proposition \ref{prop:fixedpointsshadow}, these isomorphisms satisfy the shadow coherence conditions.
 By observing the interaction of the norm functor with the geometric fixed points functor, we see that the following diagram commutes.
 $$\xymatrix{
\Phi^{C_p} \THH^{(p^n)}(A; |\Bar(M;B;L)|) \ar[d]^-\cong \ar[r]^-{\Phi^{C_p} \sigma^n _{A,B}} & \Phi^{C_p} \THH^{(p^n)}(B; |\Bar(L;A;M)|) \ar[d]^-\cong \\
 \THH^{(p^{n-1})}(A; |\Bar(M;B;L)|) \ar[r]^-{\sigma^{n-1} _{A,B}} & \THH^{(p^{n-1})}(B; |\Bar(L;A;M)|)
 }$$
 It follows that the isomorphisms $\tau^n _{A,B}$ respect the restriction maps, and therefore induce equivalences 
 $$\theta_{A,B}: \TR(A; |\Bar(M;B;L)|) \to \TR(B; |\Bar(L;A;M)|)$$
 as required. 
\end{proof}

In particular, for every $k$, $\pi_k \TR(R; -)$ extends to a shadow on the bicategory $\sR_{\cat Sp}$, taking values in $\cat{Ab}$.

\begin{prop}\label{prop-tr-lax} The Eilenberg--MacLane spectrum functor $H: s\cat{Ab} \to \cat{Sp}$ defines a lax shadow functor between the shadows $\pi_k \TR$ and $tr_k$. In particular, for any ring $A$ there is a natural linearization map 
$$\pi_k \TR(HA) \to tr_k(A).$$
\end{prop}

\begin{proof} 
By Proposition \ref{prop-equiv-EM-lax}, the equivariant Eilenberg--MacLane spectrum functor defines a lax shadow functor $\pi_k \THH_{C_{p^n}}(-;-)^{C_{p^n}} \to \m{\HH}^{C_{p^n}}(-;-)_k(C_{p^n}/C_{p^n})$. Precomposing with the norm $N^{C_{p^n}} _e: s\cat{Ab} \to s\cat{Mack}_{C_{p^n}}$, we obtain a lax shadow functor
$$\pi_k \THH_{ C_{p^n}}(H N^{C_{p^n}} _e A; H N^{C_{p^n}} _e M)^{C_{p^n}} \to \m{\HH}_k^{C_{p^n}}(N^{C_{p^n}} _e A; N^{C_{p^n}} _e M)(C_{p^n}/C_{p^n})$$
for each simplicial ring $A$ and bimodule $M$. For a Mackey functor $\m{M}$,  $N_e ^{C_{p^n}} \m{M} = \m{\pi}_0 ^{C_{p^n}} (N_e ^{C_{p^n}} H \m{M})$. If we denote $N_e ^{C_{p^n}} H \m{M}$ by $R$, the linearization map $R \to H\m{\pi}_0^{C_{p^n}}(R)$ therefore results in a natural transformation $ N^{C_{p^n}} _e H \to  H N^{C_{p^n}} _e $. 
Thus we obtain a lax shadow functor
$$\pi_k \THH_{ C_{p^n}}(N^{C_{p^n}} _e HA; N^{C_{p^n}} _e HM)^{C_{p^n}} \to \m{\HH}_k^{C_{p^n}}(N^{C_{p^n}} _e A; N^{C_{p^n}} _e M)(C_{p^n}/C_{p^n}),$$
or, simplifying notation,
$$\pi_k \THH^{(p^n)}(HA; HM)^{C_{p^n}} \to \m{\HH}_{e}^{C_{p^{n}}}(A;M)_k (C_{p^n}/C_{p^n}).$$
Since these maps are compatible with the restriction maps defined on both sides, there is an induced lax shadow functor 
$$\pi_k \holim_{\substack{\longleftarrow\\R}} \THH^{(p^n)}(HA; HM)^{C_{p^n}} \to \lim_{\substack{\longleftarrow\\R}} \pi_k \THH^{(p^n)}(HA; HM)^{C_{p^n}} \to \lim_{\substack{\longleftarrow\\r}} \m{\HH}_{e}^{C_{p^{n}}}(A;M)_k(C_{p^{n}}/C_{p^{n}}),$$
i.e., a lax shadow functor
$$\pi_k \TR(HA; HM) \to tr_k(A,M)$$
as desired.
\end{proof}

It follows from the proposition above that there is a linearization map relating TR and its algebraic analogue, $tr$, for any $(-1)$-connected ring spectrum.

\begin{cor}\label{cor-lin-map}
There is a linearization map
$$\pi_k \TR(R) \to tr_k(\pi_0(R))$$
for $(-1)$-connected ring spectra $R$.
\end{cor}

\begin{proof}
For $R$ a $(-1)$-connected ring spectrum, there is a map of ring spectra $R \to H\pi_0(R)$. We then compose with the lax shadow functor of Proposition \ref{prop-tr-lax} to obtain
$$\pi_k \TR(R) \to \pi_k \TR(H\pi_0(R)) \to tr_k(\pi_0(R)).$$
\end{proof}
In \cite{BlGeHiLa}, the authors compute the algebraic $tr$ of $\mathbb{F}_p$, proving that 
\[tr_k(\mathbb{F}_p;p) = \begin{cases} \Z_p &: k=0 \\ 0 &: \text{otherwise}, \end{cases}
\]
which agrees with $\pi_k(\TR(\mathbb{F}_p;p))$. In general, one would like to understand how closely algebraic restriction homology approximates the analogous topological theory, TR. The highly structured linearization map we construct in Proposition \ref{prop-tr-lax} is useful in pursuing this question. Below, we prove that the linearization map is often an isomorphism in degree 0.

\begin{prop}\label{prop-holim}
Let $A$ be a ring, and let $R^1 \lim$ denote the first right derived functor of the limit functor on abelian groups. If 
$$R^1 \lim_{\substack{\longleftarrow\\R}} \pi_1 \THH^{(p^n)}(HA)^{C_{p^n}} = 0,$$
then 
$$\pi_0 \TR(HA) \cong tr_0(A).$$
\end{prop}

\begin{proof}
For a ring $A$, by Theorem 5.1 of \cite{BlGeHiLa}, the map 
$$\m{\pi}_0^{C_{p^n}} \THH_{ C_{p^n}}(N^{C_{p^n}} _e HA) \to \m{\HH}^{C_{p^n}}(N^{C_{p^n}} _e A)_0$$
is an isomorphism. Therefore the map
$$\lim_{\substack{\longleftarrow\\R}} \pi_0 \THH^{(p^n)}(HA)^{C_{p^n}} \to \lim_{\substack{\longleftarrow\\r}} \m{\HH}_{e}^{C_{p^{n}}}(A)_0(C_{p^{n}}/C_{p^{n}})$$
is an isomorphism. If 
$$R^1 \lim_{\substack{\longleftarrow\\R}} \pi_1 \THH^{(p^n)}(HA)^{C_{p^n}} = 0,$$
then
$$\pi_0 \TR(HA) = \pi_0 \holim_{\substack{\longleftarrow\\R}} \THH^{(p^n)}(HA)^{C_{p^n}} = \lim_{\substack{\longleftarrow\\R}} \pi_0 \THH^{(p^n)}(HA)^{C_{p^n}},$$
and the conclusion follows.
\end{proof}

Proposition \ref{prop-holim} is analogous to the classical result that the linearization map 
\[
\pi_k(\THH(A)) \to \HH_k(A)
\]
is an isomorphism in degree 0. Classically, the linearization map on THH is also an isomorphism in degree 1. It would be interesting to consider under what conditions the linearization map on TR is similarly an isomorphism in degree 1, although we will not pursue that here.  

%% file: trace.tex

\section{Generalized Dennis traces}\label{sec:trace}

In this section, we show that there is a Dennis trace map whose target is $\THH_{C_n}(R)$. Furthermore, we show that this map factors through $\TR_{C_n}(R)$.

\subsection{The twisted Dennis trace.}

Let $(\mathsf V, \otimes, S)$ be a symmetric monoidal, simplicial model category.  As in section \ref{sec:equivHochschild}, let $(\Vaut, \otimes, (S, \textup{Id}_S))$ denote the symmetric monoidal category with objects $(V, \varphi_V)$, where $V$ is an object of  $\cat V$, and $\varphi_V$ is an automorphism of $V$, while morphisms in $\Vaut$ are morphisms in $\mathsf V$ that commute with the automorphism.

Let $(R, \varphi)$ be a monoid in $\Vaut$, and let $\Perf_{(R, \varphi)}$ denote its category of perfect left  modules in $\Vaut$. The objects in this category are perfect $R$-modules $P$ equipped with an automorphism $\gamma: P \to P$ such that the following diagram commutes:

$$\xymatrix{
R \otimes P \ar[r]^-{act} \ar[d]^-{\varphi \otimes \gamma} & P \ar[d]^-{\gamma} \\
R \otimes P \ar[r]^-{act} & P.
}$$

That is, $\gamma$ is an isomorphism of left modules from $P$ to ${}^{\varphi} P$. Here ${}^{\varphi} P$ indicates that the left action on $P$ is twisted by $\varphi: R \to R$, as in Definition \ref{def:phi-twisted_HH}. Let $\K(R, \varphi)$ denote $\K(\Perf_{(R, \varphi)})$.

Let $M$ be an $R$-bimodule. Consider the category of $M$-parametrized endomorphisms of perfect $R$-modules, denoted by Betley \cite{Bet} as $\End(R,M)$. Its objects are maps of left modules $\gamma: P \to M \otimes_R P$, where $P$ is a perfect left $R$-module, and its morphisms are maps of left modules $P \to Q$ that make the relevant diagram commute.

\begin{thm}\label{thm-dennis-trace}
There is a trace map $\K(R, \varphi) \to \HH(R, {}^{\varphi} R)$, given by the composite
$$\K(\Perf_{(R, \varphi)}) = \K(\Aut(R,{}^{\varphi} R)) \to \K(\End(R, {}^{\varphi} R)) \to \HH(R, {}^{\varphi} R).$$
When $R$ is a $C_n$-ring spectrum and $\varphi$ is given by the action of a chosen generator $g \in C_n$, we obtain a trace map $$ \K(\Perf_{(R, g)}) \to \THH_{C_n}(R).$$
\end{thm}
\begin{proof}

Observe that $\End(R, {}^{\varphi} R)$ is very similar to $\Perf_{(R, \varphi)}$, except that in $\Perf_{(R, \varphi)}$, $\gamma$ is required to be an isomorphism. We may therefore denote $\Perf_{(R, \varphi)}$ by $\Aut(R,{}^{\varphi} R)$, and it includes naturally into the category $\End(R, {}^{\varphi} R)$.

The K-theory of $\End(R,M)$, denoted $\K(R,M)$, admits a trace map to $\HH(R,M)$ (see, e.g., Section 3 of \cite{DM}), which can be considered as a twisted version of the trace $\K(\End(R,R)) \to \HH(R)$, through which the Dennis trace $\K(R) \to \HH(R)$ factors. We consider this trace map when $M = {}^{\varphi} R$. The composite of this trace with the map from $\K(\Perf_{(R, \varphi)})$ to $\K(\End(R, {}^{\varphi} R)) $ yields a twisted Dennis trace from $\K(R, \varphi)$ to $\HH(R, {}^{\varphi} R).$

When $R$ is a $C_n$-ring spectrum, and $\varphi$ is given by the action of a chosen generator $g \in C_n$, we therefore obtain a trace map $$ \K(\Perf_{(R, g)}) \to \THH(R; {}^g R) \cong \THH_{C_n}(R).$$

\end{proof}

If $R$ is a $C_n$-ring spectrum, let $\Perf^{C_n}_R$ denote the category of perfect left modules over $R$ in $C_n$-spectra. Objects in this category are modules over $R$ in $C_n$-spectra that are perfect as $R$-modules. 

\begin{cor}
There is a trace map $\K(\Perf^{C_n}_R) \to \THH_{C_n}(R)$.
\end{cor}

\begin{proof}
Let $g$ denote a chosen generator of $C_n$. We can include $\Perf^{C_n}_R$ into $\Perf_{(R,g)}$, by sending a module $P$ to $(P,g) \in \Perf_{(R,g)}$. This is a twisted version of the functor $\Perf_R \to \End(R,R)$ that sends $P$ to $(P,id)$ and induces the map $\K(R) \to \K(\End(R,R))$. We thus obtain a trace map
$$\K(\Perf^{C_n}_R) \to \THH_{C_n}(R)$$
given by the composite
$$\K(\Perf^{C_n}_R) \to \K(\Perf_{(R, g)}) = \K(\Aut(R,{}^g R)) \to \K(\End(R,{}^g  R)) \to \THH(R; {}^g R).$$
\end{proof}

\begin{rem} \label{rem-HS}
When $R$ is an ordinary ring, we can identify the trace $\K_0(R, \varphi) \to \HH_0(R; {}^{\varphi} R)$ as sending $\gamma: P \to {}^{\varphi} P$ to its Hattori-Stallings trace, i.e.,  its bicategorical trace in the category of rings, bimodules, and bimodule maps. Consider the shadow $\HH_0$ on this bicategory.  The trace of a 2-cell $\gamma: P \otimes \mathbb{Z} \to {}^\varphi R \otimes_R P$, where we regard $P$ as an $R-\mathbb{Z}$ bimodule,  is a map
$$\HH_0(\mathbb{Z}) \to \HH_0(R; {}^\varphi R)$$
and thus an element of $\HH_0(R; {}^\varphi R)$. For more details on the Hattori-Stallings trace as a bicategorical trace, see Section 5 of \cite{PS}. More generally, Theorem 1.2 of \cite{CLMPZ} states that on $\pi_0$, this twisted Dennis trace takes an endomorphism $f: P \to {}^g P$ to its bicategorical trace $\mathbb{S} \to \THH(R; {}^g R) \cong \THH_{C_n}(R)$.
\end{rem}

This twisted Dennis trace factors through the $C_n$-twisted $\TR$ of \cite{AnBlGeHiLaMa}, $\TR_{C_n}$. The following lemma will be useful in the proof.

\begin{lem} \label{lem-THHs-agree}
Let $R$ be a $C_n$-ring spectrum, and $g \in C_n$ a generator. Let $p$ be a prime that does not divide $n$, and $k$ an integer such that $p^k \equiv 1$ (mod $n$). Then
$$\THH^{(p^k)}(R; {}^g R) \simeq \THH_{C_n} (R)$$
as $C_{p^k}$-spectra.
\end{lem}

\begin{proof}
Recall from \cite{AnBlGeHiLaMa} that $\THH_{C_n} (R) \simeq \THH_{C_{n p^k}}(N_{C_n} ^{C_{n p^k }} R)$ as $S^1$-spectra. Since $p$ does not divide $n$, $C_{np^k} \cong C_n \times C_{p^k}$; let $g'$ denote a generator of $C_{p^k}$, so that $g'g$ is a generator of $C_{np^k}$. As $C_{np^k}$-spectra,
$$\THH_{C_{n p^k}}(N_{C_n} ^{C_{n p^k }} R) \simeq \THH(N_{C_n} ^{C_{n p^k }} R; {}^{g'g}N_{C_n} ^{C_{n p^k }} R) \simeq  \THH(N_{C_n} ^{C_{n p^k }} R; {}^{g'g^{p^k}}N_{C_n} ^{C_{n p^k }} R).$$
The second equivalence holds because $p^k \equiv 1$ (mod $n$). Note that the last term is equivalent to $\THH(N_{C_n} ^{C_{n p^k }} R; {}^{g'}N_{C_n} ^{C_{n p^k }} ({}^g R))$ as a $C_{np^k}$-spectrum. We now restrict to $C_{p^k}$-spectra. The restriction functor respects smash products and geometric realization, and therefore commutes with $\THH$. In addition, if $K,H$ are finite groups and $X$ is a $K$-spectrum, examining the indexed smash product defining the norm yields $\iota_{e \times H} N^{K \times H} _{K \times e} X \cong N_e ^H \iota_e X$ as $H$-spectra. Therefore we obtain
$$\iota_{C_{p^k}}\THH(N_{C_n} ^{C_{n p^k }} R; {}^{g'}N_{C_n} ^{C_{n p^k }} ({}^g R)) \simeq \THH(N_e ^{C_{ p^k }} \iota_e R; {}^{g'}N_e ^{C_{p^k }} \iota_e ({}^g R)),$$
which is equivalent as a $C_{p^k}$-spectrum to 
$$\THH_{C_{p^k}}(N_e ^{C_{p^k}} \iota_e R; N_e ^{C_{p^k}} \iota_e ({}^g R) ) =  \THH^{(p^k)}(R; {}^g R)$$
as required.
\end{proof}

\begin{prop}\label{prop-TR-trace}
The twisted Dennis trace $\K(\Perf_{(R,g)}) \to \THH_{C_n}(R)$ factors through $\TR_{C_n}(R;p)$ for $p$ coprime to $n$.
\end{prop}

\begin{proof}
By Theorem 1.3 of \cite{CLMPZ}, the twisted Dennis trace $\K(R, {}^g R) \to \THH(R, {}^g R)$ factors through $\TR$ with coefficients, $\K(R, {}^g R) \to \TR(R, {}^g R)$. On $\pi_0$, this takes an endomorphism $f: P \to {}^g P$ to the bicategorical trace of its $k$-fold composite $f^k: P \to {}^{g^k}P$ for all $k \in \mathbb{N}$.

We claim that when $p$ is coprime to $n$, $\TR(R, {}^g R ; p) \simeq \TR_{C_n}(R; p)$. We prove this by comparing the inverse systems that form both spectra. Recall that $\TR(R, {}^g R ; p)$ is the inverse limit along the restriction maps 
$$R: \THH^{(p^k)}(R; {}^g R)^{C_{p^k}} \to \THH^{(p^{k-1})}(R; {}^g R)^{C_{p^{k-1}}}.$$
Note that by Fermat's little theorem, for $p$ coprime to $n$, $p^{\varphi(n)} \equiv 1$ (mod $n$), where $\varphi(n)$ denotes Euler's totient function. We can therefore look instead at the composition of $\varphi(n)$ restriction maps at a time,
$$\THH^{(p^{j \varphi(n)})}(R; {}^g R)^{C_{p^{j \varphi(n)}}} \to \THH^{(p^{(j-1) \varphi(n)})}(R; {}^g R)^{C_{p^{(j-1) \varphi(n)}}}.$$
By Lemma \ref{lem-THHs-agree} above, this agrees with the system
$$\THH_{C_n}(R)^{C_{p^{j \varphi(n)}}} \to \THH_{C_n}(R)^{C_{p^{(j-1) \varphi(n)}}}$$
and the limit of this system agrees with the limit of the system 
$$\THH_{C_n}(R)^{C_{p^{k+1}}} \to \THH_{C_n}(R)^{C_{p^k}}$$
which is $\TR_{C_n}(R;p)$. 
\end{proof}

\begin{rem}\label{rem-twisted-TR}
This suggests that $\TR(R, {}^g R)$ is a good generalization for $\TR_{C_n}(R)$, which also works at primes that divide $n$.
\end{rem}

\subsection{Relationship to equivariant algebraic \texorpdfstring{$\K$}{K}-theory}
We have shown above that there is a twisted Dennis trace map 
\[
\K(\Perf_{(R,g)}) \to \THH_{C_n}(R)
\] that factors through $\TR_{C_n}(R)$. A natural question to ask is whether twisted topological Hochschild homology also receives a trace map from equivariant algebraic $K$-theory. Merling \cite{Merling} defined a genuine equivariant algebraic $\K$-theory $G$-spectrum, $\K_G(R)$, for a $G$-ring $R$. There is related work on equivariant algebraic $K$-theory due to Barwick and Barwick-Glasman-Shah \cite{Barwick, BGS} and Malkiewich--Merling \cite{MM}. 

The equivariant algebraic $K$-theory $\K_G(R)$ of a $G$-ring $R$ is roughly the $\K$-theory of the category of perfect $R$-modules, which admits an action of $G$ sending $P$ to ${}^g P$. Building on Merling's work \cite{Merling}, Malkiewich and Merling showed in Theorem 1.3 of \cite{MM} that for a $G$-ring spectrum $R$, the fixed point spectrum $\K_G(R)^G$ is equivalent to the $K$-theory of a category of perfect $R$-modules on which $G$ acts semilinearly. In particular, if $G$ is $C_n$ with generator $g$, then $\K_{C_n}(R)^{C_n} \simeq \K(\Perf_{(R, g)})$. Therefore we conclude the following as a corollary of Theorem \ref{thm-dennis-trace}.

\begin{cor}\label{cor:twisted}
For $R$ a $C_n$-ring spectrum, the twisted Dennis trace is a map
$$\K_{C_n}(R)^{C_n} \to \THH_{C_n}(R).$$
\end{cor}

Work of Horev \cite{Horev} shows that twisted topological Hochschild homology can be viewed through the lens of equivariant factorization homology. In particular, Horev proves in Proposition 7.2.2 of \cite{Horev} that for a $C_n$-ring spectrum $R$, 
\[\THH_{C_n}(R) \simeq \Phi^{C_n} \int_{S^1 _{rot}} R.\] Here $\int_{S^1 _{rot}} R$ denotes the $C_n$-equivariant factorization homology of $S^1$ (with the rotation action of $C_n$) with coefficients in $R$. The $C_n$-spectrum $\int_{S^1 _{rot}} R$ is non-equivariantly equivalent to $\THH(R)$, but has a ``diagonal" action of $C_n$-- that is, $C_n$ acts on both $S^1$ and $R$. This perspective on twisted topological Hochschild homology leads to the following conjecture.

\begin{conj}
The twisted Dennis trace  of Corollary \ref{cor:twisted} arises from a $C_n$-equivariant map $\K_{C_n} (R) \to \int_{S^1 _{rot}} R$.
\end{conj}

 Work in progress of Angelini-Knoll, Merling, and P\'eroux \cite{AKMP} may shed light on this conjecture, as they expand the definition of norms for compact Lie groups. Take $R$ a $C_n$-ring spectrum; using such norms, one would like to construct a Dennis trace map $\K_{C_n}(R) \to N_{e \times C_n} ^{S^1 \times C_n} R$ which is $S^1 \times C_n$ equivariant. Here we extend the $C_n$-action on $K_{C_n}(R)$ to an $S^1 \times C_n$ action using the trivial $S^1$-action. Let $H \cong C_n$ denote the diagonal subgroup of $ C_n \times C_n \leq S^1 \times C_n$. An equivariant Dennis trace of this form would produce a map
$$\K_{C_n}(R)^{C_n} = \K_{C_n}(R)^H \to (N_{e \times C_n} ^{S^1 \times C_n} R)^H \to \Phi^H N_{e \times C_n} ^{S^1 \times C_n} R.$$
If the diagonal formulas of Proposition 2.19 of \cite{AnBlGeHiLaMa} still hold for this generalized norm, so that
$$\Phi^H N_{e \times C_n} ^{S^1 \times C_n} R \cong N_{e \times C_n /(e \times C_n \cap H)} ^{S^1 \times C_n / H} \Phi^{e \times C_n \cap H} R \cong N_{C_n} ^{S^1} R,$$
then a twisted Dennis trace 
$$\K_{C_n}(R)^{C_n} \to \THH_{C_n}(R)$$ 
would indeed arise from an equivariant Dennis trace 
$$\K_{C_n}(R) \to N_{e \times C_n} ^{S^1 \times C_n} R.$$
This norm, $N_{e \times C_n} ^{S^1 \times C_n} R$, can be thought of as $\THH(R)$ taken in the category of $C_n$-spectra. That is, it has a $C_n$-action coming from the fact that $R$ is a $C_n$-ring spectrum, and an $S^1$-action coming from the cyclic bar construction. The restriction of $N_{e \times C_n} ^{S^1 \times C_n} R$ to a $C_n$-spectrum (via the diagonal subgroup $H$) should agree with $\int_{S^1 _{rot}} R$.